\documentclass[11pt]{amsart}

\usepackage{caption}
\usepackage{rotating}
\usepackage{makecell}
\usepackage{amssymb}
\usepackage{mathrsfs} % \mathscr{A} produces a script A
\usepackage{tikzexternal}
{dgs}    % tell tikz the file name to look for
\usepackage[latin1]{inputenc}
\usepackage[colorlinks=true, citecolor=blue]{hyperref} 
\allowdisplaybreaks[1] %allows line breaks in multiline equations setting argument to 4 is least restrictive

\makeatletter
\def\@tocline#1#2#3#4#5#6#7{\relax
  \ifnum #1>\c@tocdepth % then omit
  \else
    \par \addpenalty\@secpenalty\addvspace{#2}%
    \begingroup \hyphenpenalty\@M
    \@ifempty{#4}{%
      \@tempdima\csname r@tocindent\number#1\endcsname\relax
    }{%
      \@tempdima#4\relax
    }%
    \parindent\z@ \leftskip#3\relax \advance\leftskip\@tempdima\relax
    \rightskip\@pnumwidth plus4em \parfillskip-\@pnumwidth
    #5\leavevmode\hskip-\@tempdima
      \ifcase #1
       \or\or \hskip 1em \or \hskip 2em \else \hskip 3em \fi%
      #6\nobreak\relax
    \dotfill\hbox to\@pnumwidth{\@tocpagenum{#7}}\par% <---- \dotfill -> \hfill
    \nobreak
    \endgroup
  \fi}
\makeatother
\newtheoremstyle{fancy}{}{}{\itshape}{}{\textbf\bgroup}{.\egroup}{ }{}
\newtheoremstyle{fanci}{}{}{\rm}{}{\textbf\bgroup}{.\egroup}{ }{}
\newtheoremstyle{ghost}{}{}{\itshape}{}{\textbf\bgroup}{\egroup}{ }{}
\theoremstyle{fancy}
\numberwithin{equation}{section} \swapnumbers
\newtheorem{cor}[equation]{Corollary}
\newtheorem{lem}[equation]{Lemma}
\newtheorem{prop}[equation]{Proposition}
\newtheorem{thm}[equation]{Theorem}

\newtheorem*{LocallyHomogeneous}{Simply-Connected Homogeneous Three-Manifolds Admitting Compact Quotients}
\theoremstyle{fanci}

\newtheorem{dfn}[equation]{Definition}

\newtheorem{exa}[equation]{Example}

\newtheorem{rem}[equation]{Remark}

\newtheorem{nota}[equation]{Notation}

\newtheorem*{Structure}{Structure of the Paper}

\newtheorem{comp}[equation]{Computation}

\newcommand{\cref}[1]{Corollary~\ref{#1}}   %use: \cref{labelname}
\newcommand{\Z}{\mathbb Z}  % the integrers
\newcommand{\R}{\mathbb R} % the reals
\newcommand{\N}{\mathbb N} % the natural numbers
\newcommand{\Orthog}{\operatorname{O}}
\newcommand{\SO}{\operatorname{SO}}
\newcommand{\ad}{\operatorname{ad}}
\newcommand{\Ad}{\operatorname{Ad}}

\newcommand{\SL}{\operatorname{SL}}

\newcommand{\SU}{\operatorname{SU}}

\newcommand{\Inner}{\langle \cdot , \cdot \rangle}
\newcommand{\BiInner}{\langle \cdot , \cdot \rangle_0}

\newcommand{\Isom}{\operatorname{Isom}}
\newcommand{\Diff}{\operatorname{Diff}}

\newcommand{\Ric}{\operatorname{Ric}}
\newcommand{\Scal}{\operatorname{Sc}}
\newcommand{\vol}{\operatorname{vol}}

\newcommand{\germ}{\mathfrak}

\renewcommand{\phi}{\varphi}
\renewcommand{\emptyset}{\varnothing}

\newcommand{\trace}{\operatorname{trace}}  %the trace
\newcommand{\Span}{\operatorname{Span}}

\newcommand{\vecv}{{\bf v}}
\newcommand{\vece}{{\bf e}}

\begin{document}

\newcommand{\spacing}[1]{\renewcommand{\baselinestretch}{#1}\large\normalsize}
\spacing{1.14}

\title{Geometric Structures and the Laplace Spectrum, Part II}
%\title{Detecting the Moments of Inertia of a Molecule via its Rotational Spectrum}

%%%%%%%%%%%%%% CRAIG'S INFO  %%%%%%%%%%%%%
\author[S. Lin]{Samuel Lin}
\address{Dartmouth College\\ Department of Mathematics \\ Hanover, NH 03755}
\email{samuel.lin@dartmouth.edu}
\author[B. Schmidt]{Benjamin Schmidt}
\address{Michigan State University\\ Department of Mathematics \\ E. Lansing, MI 48824}
\email{schmidt@math.msu.edu}
\author[C. Sutton]{Craig Sutton$^\sharp$}
\address{Dartmouth College\\ Department of Mathematics \\ Hanover, NH 03755}
\email{craig.j.sutton@dartmouth.edu}
\thanks{$^\sharp$ The third named author was partially supported by a Simons Foundation Collaboration Grant}

%%\date{\today}

%%%%%%%%%%%%%%%%%%%%%%%%%%%%%%%%%%%%%%%%%%%%%%%%%%
%%%%%%%%%%%%%%%%%%%%%%   ABSTRACT  %%%%%%%%%%%%%%%%%%%%%
%%%%%%%%%%%%%%%%%%%%%%%%%%%%%%%%%%%%%%%%%%%%%%%%%%

\begin{abstract}
We continue our exploration of the extent to which the spectrum encodes the local geometry of a locally homogeneous three-manifold and find that if $(M,g)$ and $(N,h)$ are a pair of locally homogeneous, locally non-isometric isospectral three-manifolds, where $M$ is an elliptic three-manifold, then $(1)$ $N$ is also an elliptic three-manifold, $(2)$ $M$ and $N$ have fundamental groups of different orders, $(3)$ $(M,g)$ and $(N,h)$ both have non-degenerate Ricci tensors and $(4)$ the metrics $g$ and $h$ are sufficiently far from a metric of constant sectional curvature. We are unaware of any such isospectral pair and such a pair could not arise via the classical Sunada method. 
As part of the proof, we provide an explicit description of the isometry group of a compact simple Lie group equipped with a left-invariant metric---improving upon the results of Ochiai-Takahashi and Onishchik---which we use to classify the locally homogeneous metrics on an elliptic three-manifold $\Gamma \backslash S^3$ and we determine that any collection of isospectral locally homogeneous metrics on an elliptic three-manifold consists of at most two isometry classes that are necessarily locally isometric. In particular, the left-invariant metrics on $\SO(3)$ (respectively, $S^3$) can be mutually distinguished via their spectra. The previous statement has the following interpretation in terms of physical chemistry: the moments of inertia of a molecule can be recovered from its rotational spectrum.
%As part of the proof, we use our explicit description of the isometry group of a simple Lie group equipped with a left-invariant metric to classify the locally homogeneous metrics on an elliptic three-manifold $\Gamma \backslash S^3$ and determine that any collection of isospectral locally homogeneous metrics on this space consists of at most two isometry classes that are necessarily locally isometric. In particular, the left-invariant metrics on $\SO(3)$ (respectively, $S^3$) can be mutually distinguished via their spectra. The previous statement has the following interpretation in terms of physical chemistry: the moments of inertia of a molecule can be recovered from its rotational spectrum.
\end{abstract}

\maketitle

\tableofcontents

\setcounter{section}{0}

%%%%%%%%%%%%%%%%%%%%%%%%%%%%%%%%%%%%%%%%%%%
%%%%%%%%%%%%%             Introduction            %%%%%%%%%%%%%%%%
%%%%%%%%%%%%%%%%%%%%%%%%%%%%%%%%%%%%%%%%%%%
\section{Introduction}\label{sec:introduction}
 
In Klein's Erlangen program the simply-connected $n$-dimensional Riemannian manifolds $\mathbb{E}^n$, $\mathbb{S}^n$  and $\mathbb{H}^n$ of constant sectional curvature zero, one and negative one, respectively, correspond to the maximal geometries $(\R^n, \Isom(\mathbb{E}^n)^o)$, $(S^n, \Isom(\mathbb{S}^n)^o)$ or $(H^n, \Isom(\mathbb{H}^n)^o)$, respectively. In dimension two, these geometries account for all the maximal geometries and the Uniformization Theorem states that every closed surface $\Sigma$ admits geometric structures (i.e., complete locally homogeneous metrics) and the geometric structures supported by $\Sigma$ are all modeled on the same maximal two-dimensional geometry. Phrased differently, every closed surface $\Sigma$ admits metrics of constant sectional curvature and the sign (plus, zero or minus) of the sectional curvature is determined by the topology of the underlying surface. This demonstrates the importance of locally homogeneous metrics to the study of the geometry and topology of compact surfaces. And, it is a well-known result of Berger that a closed surface of constant sectional curvature---i.e., a compact locally homogeneous surface---is determined up to local isometry by its spectrum; specifically, its first three heat invariants \cite[Theorem 7.1]{Berger}. The numerous examples of isospectral Riemann surfaces demonstrate that this result is optimal. 

In dimension three, the simply-connected homogeneous spaces admitting compact quotients come in the following eight families.

\begin{LocallyHomogeneous}[cf. Theorem B of \cite{Sekigawa}]\label{thm:SimplyConnectedLocallyHomogeneous} 
A simply-connected homogeneous Riemannian three-manifold that admits a compact quotient is isometric to one of the following:
\begin{enumerate}
\item[(H1)] $\widetilde{E(2)} = \R^2 \rtimes_{\phi} \R$, the universal covering of the isometry group of $\mathbb{E}^2$, equipped with a left-invariant metric, where $\phi(\theta)$ is counterclockwise rotation of the plane through $2\pi\theta$; 
\item[(H2)] $S^3$ equipped with a left-invariant metric;
\item[(H3)] $\mathbb{H}^3$ equipped with a hyperbolic metric;
\item[(H4)] $S^2 \times \R$ equipped with a product of a round metric with the standard Euclidean metric;
\item[(H5)] $\mathbb{H}^2 \times \R$ equipped with a product of a hyperbolic metric with the standard euclidean metric;
\item[(H6)] $\operatorname{Nil}$ equipped with a left-invariant metric;
\item[(H7)] $\widetilde{\operatorname{SL}_2(\R)}$ equipped with a left-invariant metric;
\item[(H8)] $\operatorname{Sol}$ equipped with a left-invariant metric.
\end{enumerate}
\end{LocallyHomogeneous}
  
 \noindent
 And, the Geometrization Theorem---proven through the efforts of Hamilton \cite{Ham82, Ham86, Ham95, Ham99} and Perelman \cite{Perelman1, Perelman2, Perelman3} (cf. \cite{KL})---illuminates the importance these geometries play in understanding the geometry and topology of closed three-manifolds. 
 
Partially inspired by Berger's observation that the spectrum determines the universal Riemannian cover (i.e., determines the local geometry) of a closed surface of constant sectional curvature, we previously initiated the exploration of the extent to which compact locally homogeneous three-manifolds are determined up to local isometry by their spectra \cite{LSS}. We found that any closed three-manifold modeled on (1) a simply-connected symmetric space, (2) $\widetilde{E(2)}$ equipped with a left-invariant metric (which includes the Euclidean metric on $\R^3$), or (3) $\operatorname{Nil}$ equipped with a left-invariant metric is determined up to local isometry by its spectrum---more specifically, its first four heat invariants---among compact locally homogeneous three-manifolds \cite[Theorem 1]{LSS}. In particular, we found that six of the eight Thurston geometries (see \cite{Thurston, Scott}) are ``audible'' when compared to compact locally homogeneous three-manifolds.  

Additionally, recalling that a closed locally homogeneous three-manifold modeled on the (metrically maximal) $S^3$-geometry is defined to be a space whose universal Riemannian cover is isometric to $S^3$ equipped with a left-invariant metric \cite{LSS}, we found that the property of being a closed three-manifold modeled on the $S^3$-geometry is audible among compact locally homogeneous three-manifolds. Furthermore, if a closed Riemanian three-manifold modeled on the $S^3$-geometry is sufficiently close to a metric of constant positive sectional curvature, then it is determined up to local isometry by its spectrum within the class of locally homogeneous three-manifolds. 

In this article we delve deeper into the audibility of the local geometry of a space modeled on the $S^3$-geometry; i.e., the locally homogeneous elliptic three-manifolds.

%%%%%%%%%%%%%%%%%%%%%
\subsection{The classification of locally homogeneous elliptic three-manifolds.}
A result of Singer implies that the universal Riemannian cover of any \emph{complete} locally homogeneous manifold must be a homogeneous space \cite[p. 692]{Singer}. It then follows from the classification of simply-connected homogeneous three-manifolds admitting a compact quotient (see p.~\pageref{thm:SimplyConnectedLocallyHomogeneous}) that the Riemannian universal cover of a closed locally homogeneous three-manifold $(M,g)$ with a finite fundamental group is isometric to $S^3$ equipped with a left-invariant metric $\widetilde{g}$. That is, the locally homogeneous three-manifolds with finite fundamental group are precisely the closed three-manifolds modeled on the $S^3$-geometry. Therefore, the underlying manifold $M$ is diffeomorphic to $\Gamma \backslash S^3$, where $\Gamma$ is a finite subgroup of $\Isom(S^3, \widetilde{g})$ acting freely. From Onishchik's work on the isometry groups of homogeneous spaces \cite[Theorems 4 and 6]{Onishchik} (see Theorem~\ref{thm:IsometryGroupofSimpleLieGroup}), one concludes that $\Isom(S^3, \widetilde{g}) \leq \Isom(\mathbb{S}^3) = \operatorname{O}(4)$. Hence, $M$ is a closed three-manifold admitting a metric of positive sectional curvature, and we deduce that the closed three-manifolds admitting geometric structures modeled on an $S^3$-geometry are precisely the elliptic three-manifolds, which have been classified by Seifert and Threlfall \cite{SeifertThrelfall}.

Indeed, by definition, an elliptic three-manifold is any quotient of $S^3$ by a finite subgroup $\Gamma$ of $\SO(4)$ acting freely, and any two elliptic three-manifolds $\Gamma_1 \backslash S^3$ and $\Gamma_2 \backslash S^3$ are diffeomorphic if and only if $\Gamma_1$ and $\Gamma_2$ are conjugate in $\operatorname{O}(4)$. The finite subgroups of $\SO(4)$ acting freely on $S^3$ fall into six infinite families, Type I - VI (see Definition~\ref{dfn:GroupTypes}), where the Type I groups account for all the finite cyclic subgroups, the Type II (respectively, Type IV) groups are the product of particular subgroups of dihedral (respectively, tetrahedral) groups with certain cyclic groups, and each group of Type III (respectively, Type V or  Type VI) is a product of the binary dihedral (respectively, binary octahedral or binary icosahedral) group with a cyclic group of relatively prime order. Letting $R(\theta)$ denote counterclockwise rotation of the euclidean plane through an angle of $\theta$, one can show that up to conjugation in $O(4)$, the Type I groups which act freely on $S^3$ are given by the cyclic groups $\Gamma_{q;1,p} \equiv R(\frac{2\pi}{q}) \oplus R(\frac{2\pi p}{q})$, where $q \geq 1$ and $p$ are relatively prime integers, and the corresponding quotients $L(q;1,p) \equiv \Gamma_{q;1,p} \backslash S^3$ are the three-dimensional \emph{lens spaces}\label{text:LensSpace}. Two lens spaces $L(q_1; 1, p_1)$ and $L(q_2; 1, p_2)$ are diffeomorphic if and only if $q_1 = q_2$ and $p_2 \equiv \pm p_1, \pm p_1^{-1} \mod q_1$.

Up to scaling, every elliptic three-manifold admits a unique metric of constant positive sectional curvature. The following classification theorem---which relies on our explicit description of the isometry group of a left-invariant metric on a compact simple Lie group (see Theorem~\ref{thm:IsometryGroupofSimpleLieGroup2})---establishes that every elliptic three-manifold admits locally homogeneous metrics of non-constant sectional curvature (see Figure~\ref{fig:LocHomogEllipticThreeManifolds}). Interestingly, when $q \geq 3$ and $p \equiv \pm 1 \mod q$, the lens space $L(q;1,p)$ admits pairs of locally isometric locally homogeneous metrics where one is homogeneous while the other is not.

%\begin{center}
%\begin{sidewaysfigure}
\begin{figure}
\begin{tabular}{ | l || c | c | c |} \hline
& \makecell{$[g]$ is Berger with\\ Constant Curvature}  & \makecell{$[g]$ is Berger with\\ Non-Constant Curvature} & $[g]$ is generic \\ \hline \hline 
\makecell{$\Gamma$ is Trivial\\ or $\Z_2$} & \makecell{unique isometry class\\homogeneous} & \makecell{unique isometry class\\homogeneous} & \makecell{unique isometry class\\homogeneous} \\ \hline
\makecell{$\Gamma$ is $\Gamma_{q;1,p}$ \\ for $q \geq 3$ and $p \not\equiv \pm 1$} & \makecell{unique isometry class\\not homogeneous} & \makecell{two isometry classes\\ neither is homogeneous} & \makecell{none} \\ \hline
\makecell{$\Gamma$ is $\Gamma_{q;1,p}$\\ for $q \geq 3$ and $p \equiv \pm 1$} & \makecell{unique isometry class\\homogeneous} & \makecell{two isometry classes\\one is homogeneous} & \makecell{unique isometry class\\not homogeneous} \\ \hline
\makecell{$\Gamma$ is Type II-VI, \\ not binary dihedral or \\binary polyhedral}  & \makecell{unique isometry class\\ not homogeneous} & \makecell{unique isometry class\\ not homogeneous} & \makecell{none} \\ \hline
\makecell{$\Gamma$ is binary dihedral or \\ binary polyhedral} & \makecell{unique isometry class\\ not homogeneous} & \makecell{unique isometry class\\ not homogeneous} & \makecell{none} \\ \hline
\end{tabular}
\caption{The isometry classes of locally homogeneous metrics on $\Gamma \backslash S^3$. Here, $[g]$ is an isometry class of a left-invariant metric on $S^3$ and $\Gamma \leq \SO(4)$ is of Type I-VI.}
\label{fig:LocHomogEllipticThreeManifolds}
\end{figure}
%\end{sidewaysfigure}
%\end{center}

\begin{thm}[Classification of locally homogeneous elliptic three-manifolds]\label{thm:LocHomogEllipticThreeManifolds} Let $M = \Gamma \backslash S^3$ be an elliptic three-manifold. And, by a Berger metric on $S^3$ we shall mean a left-invariant metric on $S^3$ obtained from the round metric by scaling the fibers of the Hopf fibration (see Definition~\ref{dfn:BergerMetrics}).

\begin{enumerate}
\item If $\Gamma$ is Type I and trivial (respectively, isomorphic to $\Z_2$, the case when $q=2$), then $M$ is $S^3$ (respectively, $\SO(3)$) and the isometry classes of locally homogeneous metrics on $M$ are precisely the isometry classes of the left-invariant metrics on $M$. Consequently, all locally homogeneous metrics on $M$ are homogeneous.

\item If $\Gamma$ is Type I and conjugate to $\Gamma_{q;1,p}$, where $q \geq 3$ and $p \not\equiv \pm 1 \mod q$, then each isometry class of a metric of constant sectional curvature on $S^3$ covers a unique isometry class of locally homogeneous metrics on $M$ and this class is not homogeneous, while a Berger metric of non-constant sectional curvature on $S^3$ covers precisely two isometry classes of locally homogeneous metrics on $M$ neither of which is homogeneous. Any isometry class of a locally homogeneous Riemannian metric on $M$ arises in this manner. 

\item If $\Gamma$ is Type I and conjugate to $\Gamma_{q;1,p}$, where $q \geq 3$ and $p \equiv \pm 1 \mod q$, then every isometry class of a left-invariant metric on $S^3$ covers an isometry class of locally homogeneous Riemannian metrics on $M$ and any isometry class of locally homogeneous Riemannian metrics on $M$ arises in this fashion. Specifically, we have the following. 
\begin{enumerate}
\item The isometry class of a metric of constant positive sectional curvature on $S^3$ covers a unique isometry class of locally homogeneous Riemannian metrics on $M$ and this class is homogeneous. 
\item The isometry class of a Berger metric of non-constant curvature on $S^3$ covers exactly two isometry classes of locally homogeneous Riemannian metrics on $M$ and precisely one of these classes is homogeneous.
\item The isometry class of a left-invariant metric on $S^3$ that is not a Berger metric covers a unique isometry class of locally homogeneous Riemannian metrics on $M$ and this class is not homogeneous.  
\end{enumerate}

\item If $\Gamma$ is Type II-VI, but not binary dihedral or binary polyhedral, then each isometry class of a Berger metric on $S^3$ covers a unique isometry class of locally homogeneous Riemannian metrics on $M$ and this class is not homogeneous. Any isometry class of a locally homogeneous metric on $M$ arises in this manner.

\item If $\Gamma$ is binary dihedral or binary polyhedral, then each isometry class of a left-invariant metric on $S^3$ covers a unique isometry class of locally homogeneous Riemannian metrics on $M$ and this class is homogeneous if and only if it is of constant sectional curvature. Any isometry class of a locally homogeneous Riemannian metric on $M$ arises in this manner. 
   
\end{enumerate}
In particular, given an isometry class of a locally homogeneous Riemannian metric on $M$ there is at most one other isometry class on $M$ that shares the same universal Riemannian cover (up to isometry). 
\end{thm}

In summary, given an elliptic three-manifold $M = \Gamma \backslash S^3$, an isometry class of a left-invariant metric on $S^3$ covers zero, one, or two locally homogeneous isometry classes on $M$---according to the conjugacy class of $\Gamma$ inside $O(4)$---and every isometry class of a locally homogeneous metric on $M$ arises in this fashion. Every isometry class of a Berger metric on $S^3$ covers an isometry class of a locally homogeneous metric on $M$, and an isometry class of a left-invariant metric $g$ on $S^3$ covers two distinct locally homogeneous isometry classes on $M$ if and only if $g$ is a Berger metric of non-constant sectional curvature and $\Gamma$ is conjugate to $\Gamma_{q; 1, p}$ for some $q \geq 3$. When $p \not\equiv \pm 1 \mod q$ neither of the isometry classes on $L(q;1,p)$ covered by a Berger metric is homogeneous. However, when $p \equiv \pm 1 \mod q$, precisely one of the isometry classes covered by a Berger metric is homogeneous (cf. Theorem~\ref{thm:LocHomogEllipticThreeManifoldsV2}). 

After observing that any compact homogeneous three-manifold is covered by a compact homogeneous three-manifold with an abelian fundamental group (see Lemma~\ref{lem:HomogeneousQuotients}), the preceding theorem and the classification of the Thurston geometries admitting compact quotients with abelian fundamental group \cite[Tables 1 and 2]{AFW} allow us to classify the compact homogeneous three-manifolds.

\begin{thm}[Classification of compact homogeneous three-manifolds]\label{thm:HomogeneousThreeManifolds}
A compact homogeneous three-manifold is isometric to one of the following spaces:

\begin{enumerate}
\item a flat three-dimensional torus;
\item $S^2\times S^1$ with a globally symmetric metric;
\item $\R P^2 \times S^1$ with a globally symmetric metric;
\item the non-trivial $S^{1}$-bundle over $\mathbb{R}P^{2}$ equipped with a locally symmetric metric;
\item a locally homogeneous elliptic three-manifold $(\Gamma \backslash S^3, h)$ where the isometry class $[\widetilde{h}]$ of the covering metric $\widetilde{h}$ on $S^3$ and the finite group $\Gamma \leq \Isom (S^3, \widetilde{h})^o \leq \operatorname{SO}(4)$ are as follows: 
\begin{enumerate}
\item $[\widetilde{h}]$ is the isometry class of any left-invariant metric on $S^3$ and $\Gamma$ is trivial or $\{\pm I\} \simeq \Z_2$;
\item $[\widetilde{h}]$ is the isometry class of any Berger metric on $S^3$ and, up to conjugation in $\Isom(S^3, h)$, $\Gamma = \Gamma_{q;1,-1}$ , for $q >2$;

\item $[\widetilde{h}]$ is the isometry class of any metric of constant curvature on $S^3$ and $\Gamma$ is a binary dihedral or binary polyhedral group.   
\end{enumerate}
\end{enumerate}
\end{thm}

%%%%%%%%%%%%%%%%%%%%
\subsection{On hearing the local geometry of an elliptic three-manifold.}
Returning to the question of the audibility of the local geometry of a locally homogeneous elliptic three-manifold, we  note that there are currently no non-trivial examples of isospectral locally homogeneous elliptic three-manifolds. In fact, building off of Berger's observation that a closed elliptic three-manifold of constant sectional curvature is determined up to local isometry by its first three heat invariants \cite[Theorem 7.1]{Berger}, Wolf has shown that, among \emph{all} closed three-manifolds, a manifold of constant positive sectional curvature is uniquely determined by its spectrum \cite{Wolf01}. Additionally, several years ago, the second and third authors observed, in an unpublished preprint \cite{SchmidtSuttonUnpublished}, that the isometry classes of left-invariant metrics on $S^3$ (respectively, $\SO(3)$) can be mutually distinguished by their first four heat invariants (cf. \cite{Sutton} and \cite[Sec. 4]{Sutton2}). Recently, E. Lauret has made this same observation by explicitly calculating the fundamental tone of such metrics \cite{ELauret19}. 

By modifying the approach taken in \cite{SchmidtSuttonUnpublished} and appealing to Theorem~\ref{thm:LocHomogEllipticThreeManifolds}, we are able to show that on an elliptic three-manifold an isospectral set of locally homogeneous metrics consists of at most two isometry classes and these classes must be locally isometric.
 
\begin{thm}[Spectral isolation of locally homogeneous elliptic three-manifolds]\label{thm:IsospectralSets}
Any collection of locally homogeneous isometry classes on an elliptic three-manifold $\Gamma \backslash S^3$ for which the first four heat invariants agree consists of at most two classes and these classes are locally isometric. Therefore, in the event that $\Gamma$ is not conjugate to $\Gamma_{q;1,p}$ with $q \geq 3$, the isometry classes of locally homogeneous metrics on $\Gamma \backslash S^3$ can be mutually distinguished via their first four heat invariants. In particular, the isometry classes of left-invariant metrics on $S^3$ (respectively $\SO(3)$) can be mutually distinguished via their first four heat invariants.
\end{thm}

\noindent
Consequently, within the space of locally homogeneous metrics on an elliptic three-manifold, each metric possesses a neighborhood in which it is uniquely determined by its spectrum; that is, each metric is spectrally isolated. Coupled with \cite{DoyleRossetti}, \cite{Sharafutdinov} and \cite{LSS}, this presents strong evidence that one should \emph{not} expect to find isospectral deformations through locally homogeneous metrics in dimension three, which is in contrast with the many examples of non-trivial isospectral deformations through (locally) homogeneous metrics that exist in large dimension \cite{Gordon93, Gordon94, Schueth01, Proctor}. 

We conclude this section by observing that, among locally homogeneous three-manifolds, if there is an isospectral pair demonstrating that the local geometry of an elliptic three-manifold is not encoded in its spectrum, then it must take the following form. 
 
 \begin{thm}\label{thm:AudibillityLocalGeometry}
Let $(M_1, g_1)$ and $(M_2, g_2)$ be compact locally homogeneous three-manifolds that share the same first four heat invariants---i.e., $a_j(M_1, g_1) = a_j(M_2, g_2)$ for $j =0, 1,2, 3$---but are not locally isometric. If $M_1$ is an elliptic three-manifold, then 
\begin{enumerate}
\item $M_2$ is an elliptic three-manifold;
\item the fundamental groups of $M_1$ and $M_2$ are of different orders;
\item $(M_1, g_1)$ and $(M_2, g_2)$ both have non-degenerate Ricci tensors;
\item the metrics $g_1$ and $g_2$ are both sufficiently far from a metric of constant positive sectional curvature. 
\end{enumerate} 
\end{thm}

\noindent
To the best of our knowledge, there are no known pairs of isospectral locally homogeneous elliptic three-manifolds with distinct Riemannian covers; in particular, classical incarnations of Sunada's method cannot produce such a pair as they would necessarily have a common Riemannian covering.

%%%%%%%%%%%%%%%%%%%%%%%
\subsection{Detecting the moments of inertia of a molecule via its rotational spectrum}
Spectral geometry enjoys rich connections with quantum mechanics and spectroscopy. We now demonstrate that when we consider the elliptic   three-manifold $\SO(3)$, Theorem~\ref{thm:IsospectralSets} has an interesting interpretation in terms of physical chemistry.

Consider a rigid three-dimensional body ${\bf W}$ with center of mass at the origin.  The \emph{moment of inertia tensor} of  ${\bf W}$ is a positive, self-adjoint linear isomorphism $\mathbb{I}: (\mathbb{R}^3,\Inner) \to (\mathbb{R}^3,\Inner)$ with respect to the Euclidean inner product $\Inner$.  The \emph{moment of inertia} of {\bf W} about an axis $\R \vecv$, where $\vecv \in S^2$, is the scalar  $\langle \mathbb{I} (\vecv) , \vecv \rangle$ and measures the resistance of ${\bf W}$ to rotation about the axis $\R \vecv$.

The moment of inertia tensor has an orthonormal eigenbasis $\{\vece_1, \vece_2, \vece_3\}$, with corresponding eigenvalues $0< I_1\leq I_2 \leq I_3$. 
The numbers $I_1, I_2,$ and $I_3$ are the \emph{principal moments of inertia} of the body and the vectors $\vece_1, \vece_2$ and $\vece_3$ are the \emph{principal axes}.
A body is \emph{spherical} when all principal moments of inertia are equal (e.g., the molecule methane), \emph{symmetric} when exactly two of the principal moments of inertia are equal (e.g., benzene and chloromethane),
and \emph{asymmetric} otherwise (e.g., water). 

The principal moments of inertia $0 < I_1 \leq I_2 \leq I_3$ determine a left-invariant metric $g_{(I_1, I_2, I_3)}$ on $\SO(3)$ as follows.  Let $B( \cdot, \cdot)$ denote the Killing form on the Lie algebra $\germ{so}(3)$ and let $\Theta_1, \Theta_2, \Theta_3$ denote the usual orthonormal basis of $\germ{so}(3)$ with respect to the inner product $-B$. The triple $0 < I_1 \leq I_2 \leq I_3$ determines a self-adjoint map $\mathbb{I}_{I_1, I_2, I_3}: (\germ{so}(3), -B) \to (\germ{so}(3), -B)$ defined by $\Theta_{j} \mapsto \frac{1}{I_j} \Theta_j$, for $j = 1, 2, 3$. Then, $g_{(I_1,I_2, I_3)}$ is the left-invariant metric on $\SO(3)$ induced by the inner product $ \langle u, v \rangle = -B(\mathbb{I}_{I_1, I_2, I_3} (u), v)$ on $\germ{so}(3)$. For example, the metric $g_{(1,1,1)}$ is the unique (up to scaling) bi-invariant metric on $\SO(3)$.  Letting $\mathcal{I} = \{ (I_1, I_2, I_3): 0 < I_1 \leq I_2 \leq I_3 \}$ and letting $\overline{\mathscr{R}}_{\rm{left}}(\SO(3))$ denote the space of isometry classes of left-invariant metrics on $\SO(3)$, Proposition~\ref{prop:MetricEigenvaluesSO(3)} implies that the map $\mathcal{I} \rightarrow \overline{\mathscr{R}}_{\rm{left}}(\SO(3))$ defined by $(I_1, I_2, I_3) \mapsto g_{(I_1,I_2, I_3)}$ is a bijection. 
 
Classical mechanics implies that the geodesics in $\SO(3)$ with respect to the left-invariant metric $g_{(I_1,I_2, I_3)}$ describe free rotations of  {\bf W} about its center of mass (cf. \cite[Section 28]{GuSt}).  When {\bf W} is a molecule, Schr\"{o}dinger's equation implies that the eigenvalues associated to the Laplacian of $g_{(I_1,I_2, I_3)}$ describe the rotational spectrum (or energy levels) of the molecule.  Applying Theorem~\ref{thm:IsospectralSets} to the class of left-invariant metrics on $\SO(3)$ yields the following statement:
\begin{center}
\emph{The rotational spectrum of a molecule determines its moments of inertia.}
\end{center}

\noindent
This was previously observed by the second and third authors in the unpublished articles \cite{Sutton} and \cite{ SchmidtSuttonUnpublished} (cf. \cite[Sec. 4]{Sutton2}).

%%%%%%%%%%%%%%%%%%%%%
\subsection{Recovering a metric from the first three heat invariants} Returning to geometry, in light of Theorem~\ref{thm:IsospectralSets}, it is natural to wonder whether it is possible to mutually distinguish the locally homogeneous metrics on an elliptic three-manifold with only the first three heat invariants. On a locally homogeneous space $(M,g)$ the first three heat invariants are given by $a_0(M,g) = V$, $a_1(M,g) = \frac{1}{6} V \cdot \Scal$ and $a_2(M,g) = \frac{1}{360} V (2 (|R|^2 - |\rho|^2) + 5\Scal^2)$, where $V$, $\Scal$, $R$ and $\rho$ denote the volume, scalar curvature, curvature tensor and Ricci tensor of $g$, respectively. Although the volume and the curvature tensor determine the heat invariants $a_0(M,g)$, $a_1(M,g)$ and $a_2(M,g)$, these heat invariants need not determine the isometry class of $g$ \cite{SchmidtSuttonUnpublished}. Nevertheless, $a_0(M,g)$, $a_1(M,g)$ and $a_2(M,g)$ suffice in some cases.

\begin{thm}\label{thm:A0A1A2SufficeScalarFlat}
Let $(\Gamma \backslash S^3, g)$ be a locally homogeneous elliptic three-manifold. 
Now, suppose one of the following holds:
\begin{enumerate}
\item $g$ has non-positive scalar curvature, or
\item $g$ has positive scalar curvature and 
$27a_1(\Gamma \backslash S^3, g)^2 - 30a_0(\Gamma \backslash S^3, g)a_2(\Gamma \backslash S^3, g) \geq 0$.
\end{enumerate}
Then, the first three heat invariants determine the isometry class of $g$ up to local isometry among all locally homogeneous metrics on $\Gamma \backslash S^3$. In particular, when $\Gamma$ is trivial or $\mathbb{Z}_2$, such metrics are uniquely determined by their first three heat invariants among all left-invariant metrics.
\end{thm}

\noindent
In particular, we remark that  a locally homogeneous elliptic three-manifold with degenerate Ricci tensor is determined up to local isometry by its first three heat invariants. Indeed, the equality $27a_1(M,g)^2 - 30a_0(M,g)a_2(M,g) = 0$ is equivalent to the metric possessing a degenerate Ricci tensor (see Corollary~\ref{cor:DegenerateRicci}) and all locally homogeneous elliptic three-manifolds with degenerate Ricci tensor must have positive scalar curvature (Lemma~\ref{lem:ScalVolDegenerateRicci}) and we apply the theorem.
 
\subsection{Recovering a metric from its curvature tensor} Two Riemannian homogeneous manifolds $(M,g)$ and $(\widehat{M}, \hat{g})$ are said to have \emph{identical}\label{IdenticalCurvature} curvature tensors $R_1$ and $R_2$ if for each $p \in M$ and $\hat{p} \in \widehat{M}$ there is a linear isometry $F: (T_pM, g_p) \to (T_{\hat{p}}\widehat{M}, \hat{g}_{\hat{p}})$ such that $F^*\widehat{R}_{\hat{p}} = R_{p}$. In dimension $2$ it is clear that homogeneous manifolds with identical curvature tensor are locally isometric. In contrast, continuous families of locally non-isometric left-invariant metrics on $\SU(2)$ (resp. $\SO(3)$ and $\SL(2, \R)$) with identical curvature tensor are exhibited in \cite{Lastaria, ScWo1, ScWo2}. The machinery used to prove Theorem~\ref{thm:AudibillityLocalGeometry} allows us to demonstrate that for $\SU(2)$, $\SO(3)$ and any other elliptic three-manifold these ambiguities can be resolved (up to local isometry) by considering volume.

\begin{thm}\label{thm:VolCurvature} 
Locally homogeneous metrics on an elliptic three-manifold $\Gamma \backslash S^3$ with identical curvature tensor and volume are locally isometric. In the event that $\Gamma$ is not conjugate to $\Gamma_{q;1,p}$, for $q \geq 3$, we may replace ``locally isometric'' with ``isometric'' in the preceding sentence.
\end{thm}

\begin{Structure}
The purpose of Section \ref{section:homogeneous} is to classify all locally homogeneous elliptic three-manifolds up to isometry. We begin with a result by Seifert and Threlfall that classifies the groups that can appear as the fundamental group of an elliptic three-manifold up to conjugation in $O(4)$. Then, after an explicit computation of the isometry group of a left-invariant metric on a compact simple Lie group (Theorem~\ref{thm:IsometryGroupofSimpleLieGroup2}), the metric classification follows by investigating all possible ways of realizing these fundamental groups as subgroups of isometry groups of left-invariant metrics on $S^3$. Section \ref{sec:CptHomogeneous} is dedicated to the classification of all compact \textit{homogeneous} three-manifolds. We later deduce that the first four heat invariants cannot distinguish homogeneous three-manifolds among locally homogeneous ones. In Section \ref{sec:SphericalHeatInvariants}, we express the first four heat invariants in terms of the Christofffel symbols with respect to Milnor frame \cite[Definition 2.13]{LSS}, and apply the formulae to prove the main theorems. Building off the method developed in Section \ref{sec:SphericalHeatInvariants}, we demonstrate that the local isometry type of a locally homogeneous elliptic three-manifold is determined by its curvature tensor and volume in Section \ref{sec:ExtraStuff}. 
\end{Structure}

%%%%%%%%%%%%%%%%%%%%%%%%%%%%%%%%%%%%%%%%%%%%
%%%%%%%%%%%%%%Locally  Homogeneous Spherical Manifolds%%%%%%%%%
%%%%%%%%%%%%%%%%%%%%%%%%%%%%%%%%%%%%%%%%%%%%

\section{Locally homogeneous three-manifolds with finite fundamental group}\label{section:homogeneous}
Singer's observation that the universal Riemannian cover of a complete locally homogeneous manifold must be homogeneous \cite[p. 692]{Singer} coupled with the classification of simply-connected homogeneous spaces admitting compact quotients (see p.~\pageref{thm:SimplyConnectedLocallyHomogeneous}) tells us that the universal Riemannian cover of any closed locally homogeneous three-manifold with finite fundamental group is isometric to the three-sphere equipped with a left-invariant metric; that is, it is modeled on the $S^3$-geometry. And, Onishchik's work on the isometry groups of homogeneous spaces (Theorem~\ref{thm:IsometryGroupofSimpleLieGroup}) implies the isometry group of any left-invariant metric on $S^3$ is a subgroup of $\operatorname{O}(4)$. It follows that the closed three-manifolds with finite fundamental groups \emph{and} admitting geometric structures are precisely the elliptic three-manifolds (i.e., quotients of $S^3$ by finite subgroups of $\SO(4)$ acting freely): a geometric version of the elliptization theorem. The goal of this section is to verify Theorem~\ref{thm:LocHomogEllipticThreeManifolds} which provides a classification of the locally homogeneous metrics supported by the elliptic three-manifolds; i.e., the compact  locally homogeneous three-manifolds with finite fundamental group. This classification will contribute to our categorization of compact homogeneous three-manifolds given in Theorem~\ref{thm:HomogeneousThreeManifolds}. Our general strategy is as follows.

After establishing some notation in Section~\ref{sec:Notation}, we use Section~\ref{sec:TopologicalClassification} to review Seifert and Threlfall's classification of elliptic three-manifolds (Theorem~\ref{SeifertThrelfalltheorem}): the finite subgroups of $\SO(4)$ acting freely on $S^3$ come in six explicit families (Type I-VI) and the resulting quotient manifolds are diffeomorphic if and only if the subgroups are conjugate in $\operatorname{O}(4)$. By the work of Sekigawa \cite[Theorem B]{Sekigawa} (see p.~\pageref{thm:SimplyConnectedLocallyHomogeneous}), the universal cover of a locally homogeneous elliptic three-manifold is isometric to $S^3$ equipped with a left-invariant metric and it follows from a result of Onischik (Theorem~\ref{thm:IsometryGroupofSimpleLieGroup}) that the isometry group of a left-invariant metric on $S^3$ is a subgroup of $\operatorname{O}(4)$. In Section~\ref{sec:IsometryClassesOfLeftInvariantMetrics}, we observe that the left-invariant metrics on $S^3$ are determined up to isometry by their so-called metric eigenvalues (see p.~\pageref{page:MetricEigenvalues}). In fact, given a left-invariant metric $g$ on $S^3$ and $h \in S^3$, the metric $C_{h}^{*}g$---where $C_h$ is conjugation by $h$---is a left-invariant metric on $S^3$ with identical metric eigenvalues and all left-invariant metrics on $S^3$ with metric eigenvalues identical to those of $g$ arise in this manner. Therefore, up to isometry, every locally homogeneous elliptic three-manifold is of the form $(\Gamma \backslash S^3, h_\Gamma)$, where $h \in \mathscr{R}_{\rm{left}}(S^3)$, $\Gamma \leq \Isom(S^3, h)$ is of Type I-VI, and $h_\Gamma$ is the corresponding quotient metric on $\Gamma \backslash S^3$, and we see that our task is now to determine when two such quotients are isometric.

With this goal in mind, in Section~\ref{sec:IsometryGroups}, we determine which subgroups of Type I-VI in $\SO(4)$ can occur as subgroups of the isometry group of a given left-invariant metric (Corollary~\ref{cor:DiscreteSubgroupsIsometryGroupsS^3}). This is achieved by providing an effective description of the isometry group of a left-invariant metric on a compact simple Lie group (Theorem~\ref{thm:IsometryGroupSimple}) and applying it to compute the isometry groups of left-invariant metrics on $S^3$ (Proposition~\ref{prop:IsometryGroupS^3}). Now, suppose $\Gamma_1 \backslash S^3$ and $\Gamma_2 \backslash S^3$ are two diffeomorphic elliptic three-manifolds. What remains is to determine when left-invariant metrics $h_1$ and $h_2$ on $S^3$, with the property  that $\Gamma_1 \leq \Isom(S^3, h_1)$ and $\Gamma_2 \leq \Isom(S^3, h_2)$, induce isometric quotient metrics.  We solve this problem in Section~\ref{sec:LocHomogElliptic} via Theorem~\ref{thm:LocHomogEllipticThreeManifoldsV2}---a reformulation of Theorem~\ref{thm:LocHomogEllipticThreeManifolds}---which provides an explicit description of the isometry classes of the locally homogeneous elliptic three-manifolds and---through Wolf's criteria (Theorme~\ref{thm:HomogeneousCriterion})---identifies those that are actually homogeneous.

%%%%%%%%%%%%%%%%%%%%%%%%%%%%
\subsection{Some notation and a lemma}\label{sec:Notation} Let $M$ be a connected smooth manifold, $\operatorname{Diff}(M)$ its group of diffeomorphisms, and $\mathscr{R}(M)$ the space of Riemannian metrics on $M$.  The group $\operatorname{Diff}(M)$ acts on $\mathscr{R}(M)$ by pullback.  The orbit space, denoted by $\overline{\mathscr{R}}(M) := \mathscr{R}(M)/\operatorname{Diff}(M),$ is the space of isometry classes of Riemannian metrics on $M$ and, for each Riemannian metric $g \in \mathscr{R}(M)$, we let $[g] \in \overline{\mathscr{R}}(M)$ denote its isometry class.

Given a Lie group $G \leq \operatorname{Diff}(M)$ we let $\mathscr{R}_G(M)$ be the set of $G$-invariant metrics on $M$ and $\overline{\mathscr{R}}_G(M) = \{[g]\,\vert\,g\in \mathscr{R}_G(M)\}$ the isometry classes of $G$-invariant metrics on $M$. Given a discrete subgroup $\Gamma \leq \operatorname{Diff}(M)$ acting freely and properly discontinuously on $M$ we let $\mathscr{R}^{\Gamma}_G(M)$ be the set of $\Gamma$-invariant metrics in $\mathscr{R}_G(M)$. Then, $h \in \mathscr{R}^{\Gamma}_G(M)$ induces a Riemannian metric $h_\Gamma \in \mathscr{R}(\Gamma \backslash M)$. In this case we adopt the following notation: 
\begin{itemize}
\item $\mathscr{R}^{\Gamma}_{G}(\Gamma \backslash M) = \{ h_\Gamma : h \in \mathscr{R}^{\Gamma}_{G}(M) \}$,
\item $\overline{\mathscr{R}}^{\Gamma}_{G}(\Gamma \backslash M) = \{ [h_\Gamma] : h \in \mathscr{R}^{\Gamma}_{G}(M) \}$,
\item $\overline{\mathscr{R}}_{G}(\Gamma \backslash M)$ denotes the isometry classes of metrics on $\Gamma \backslash M$ \emph{covered by} a $G$-invariant metric on $M$.
\end{itemize}

\noindent 
Given $[h] \in \overline{\mathscr{R}}_{G}(\Gamma \backslash M)$ there is a $\Gamma' \leq \Diff(M)$, $g \in \mathscr{R}^{\Gamma'}_{G}(M)$ and a diffeomorphism $\phi = \phi_{\Gamma'}: \Gamma \backslash M \to \Gamma' \backslash M$ such that $\phi^{*}_{\Gamma'} g_{\Gamma'} \in [h]$. Therefore, letting $\mathcal{A}$ denote the set of discrete subgroups $\Gamma' \leq \Diff(M)$ acting properly discontinuously on $M$ with a diffeomorphism $\phi_{\Gamma'} : \Gamma \backslash M \to \Gamma' \backslash M$ and $\mathscr{R}^{\Gamma'}_{G}(M) \neq \emptyset$, we see
$$\overline{\mathscr{R}}_{G}(\Gamma \backslash M) = \cup_{\Gamma' \in \mathcal{A}} \phi^{*}_{\Gamma'} \overline{\mathscr{R}}_{G}^{\Gamma'}(\Gamma' \backslash M).$$

Let $K$ be a subgroup of a group $H$. Given any element $h\in H$, we denote the conjugate of $K$ by $h$ as
$$K^{h}=hKh^{-1}.$$
It will be useful to recall the following lemma.

\begin{lem}\label{lem:Basic}
Let $f \in \operatorname{Diff}(M)$ and $G_2=G_1^{f}$. Suppose that $\Gamma_1$ is a discrete subgroup of $\operatorname{Diff}(M)$ acting freely and properly discontinuously on $M$, and $\Gamma_2 = \Gamma_1^{f}$. Then, for any $g \in \mathscr{R}^{\Gamma_2}_{G_2}(M)$, we have $f^*g \in \mathscr{R}^{\Gamma_1}_{G_1}(M)$ and $\Isom(M, f^*g) =\Isom(M, g)^{f^{-1}}$.
Moreover, $f$ induces an isometry 
$$\overline{f}: (\Gamma_1 \backslash M, (f^*g)_{\Gamma_1}) \to (\Gamma_2 \backslash M, g_{\Gamma_2})$$
 and we have a bijection $h \in \mathscr{R}^{\Gamma_2}_{G_{2}}(\Gamma_2 \backslash M) \mapsto \overline{f}^*h \in \mathscr{R}^{\Gamma_1}_{G_1}(\Gamma_1 \backslash M)$.
\end{lem}

%%%%%%%%%%%%%%%%%%
\subsection{Classification of elliptic three-manifolds}\label{sec:TopologicalClassification} 
We turn to the classification of elliptic three-manifolds. In order to keep the present article self-contained and to establish notation, we now review Seifert and Threlfall's classification of the spherical space forms which, in light of the Elliptization Theorem, provides a classification of all closed three-manifolds with a finite fundamental group. Our exposition will draw on \cite[Sec. 2.6]{Wolf},  \cite[Sec. 3]{Ikeda80} and \cite[Chp. 6]{Orlick}.

Let $Q = \operatorname{span}_{\R}\{1 , i , j , k\}$ denote the quaternion algebra defined by $i^2=j^2=k^2=ijk=-1$. Given $$q=a+bi+cj+dk \in Q,$$ let $\bar{q}=a-bi-cj-dk$ denote its conjugate.  Define an inner-product and norm on $Q$ by
$q_1 \cdot q_2 = \operatorname{Re}(\bar{q}_1 q_2)$ and $|q_1|^2 = q_1 \bar{q}_1$ for any $q_1, q_2 \in Q$. Let $S^3 = \{ q \in Q : |q|^2 = 1 \}$ be the unit sphere.
Identifying $Q$ with $\R^4$, we let 
$L_q, R_q: Q \equiv \R^4 \to Q \equiv \R^4$
denote left and right multiplication by $q \in Q$, respectively,
and we obtain a double covering of Lie groups 
$\rho: S^3 \times S^3 \to \SO(4)$ given by 
$(q_1, q_2) \mapsto L_{q_1} \circ R_{q_2^{-1}} = L_{q_1} \circ R_{\overline{q}_2}$.
The subgroups 
\begin{itemize}
\item $L \equiv \rho (S^3 \times 1)$
\item $R \equiv \rho (1 \times S^3)$
\end{itemize}
are normal in $\SO(4)$ and for any $A \in \Orthog (4) - \SO(4)$ we have $L^A = R$. The $L$-invariant (resp., $R$-invariant ) metrics on $S^3$ are the left-invariant (resp. right-invariant) metrics on $S^3$.
One can check that $\rho(e^{i\theta}, e^{i \phi}) = R(\theta - \phi) \oplus R(\theta + \phi)$, where 
$$R(\alpha) = \left(\begin{array}{cc}\cos(\alpha) & -\sin (\alpha) \\ \sin(\alpha) & \cos(\alpha) \end{array}\right)$$ is counterclockwise rotation through the angle $\alpha$.
By the \emph{standard circle} in $S^3$ we mean 
$$S^{1}_{\rm{std}} \equiv \{e^{i\theta} : \theta \in \R \} \leq S^3.$$
Then, the \emph{standard maximal torus} $S^{1}_{\rm{std}} \times S^{1}_{\rm{std}}$  in $S^3 \times S^3$ is mapped to the standard maximal torus $T^{2}_{\rm{max}} \equiv \{ R(\alpha) \oplus R(\beta) : \alpha, \beta \in \R \}$ in $\SO(4)$.
Similarly, identifying $\R^3$ with the space $\operatorname{span}_{\R}\{ i , j , k\}$ 
of pure quaternions we obtain the two-fold covering of Lie groups $\pi: S^3 \to \SO(3)$ 
given by $\pi(q)(x) = qxq^{-1}$; i.e., $\pi = \rho \upharpoonright \Delta S^3$, where $\Delta S^3$ is the diagonal embedding of $S^3$ in $S^3 \times S^3$.

It is well known that up to conjugation the finite subgroups of $\SO(3)$ are the cyclic group $\Z_q$ of order $q$, the dihedral group ${\bf D}_k$ of order $2k$ ($k \geq 2$), the tetrahedral group ${\bf T}$ of order 12, the octahedral group ${\bf O}$ of order $24$, or the icosahedral group ${\bf I}$ of order $60$ (see \cite[Thm. 2.6.5]{Wolf}). In terms of generators and relations these groups are given by 
\begin{itemize}
\item $\Z_q = \langle A: A^q = 1 \rangle$
\item ${\bf D}_k = \langle A, B : A^k = B^2 = 1, BAB^{-1} = A^{-1} \rangle$
\item ${\bf T } =  \langle A, P, Z : A^3 = P^2 = Z^2 = 1, PZ =ZP, APA^{-1} = Z, AZA^{-1} = PZ \rangle$
\item ${\bf O } =  \langle A, P, Z, R: A^3 = P^2 = Z^2 = R^2 = 1, PZ =ZP, 
RAR^{-1} = A^{-1}, APA^{-1} = Z, AZA^{-1} = PZ, PZP^{-1} = ZP, RZR^{-1} = Z^{-1}\rangle$ 
\item ${\bf I } =  \langle A, B,C : A^3= B^2=C^5 = ABC = 1 \rangle$
\end{itemize}
and it is known that ${\bf T}$ is isomorphic to $A_4$, ${\bf O}$ is isomorphic to $S_4$ and ${\bf I}$ is isomorphic to $A_5$.

Turning to $S^3$, it is well known that two finite subgroups of $S^3$ are isomorphic if and only if they are conjugate in $S^3$ and up to conjugation the finite subgroups are \cite[Th, 2.6.7]{Wolf}: 
\begin{itemize} 
\item ${\bf C}_d = \langle e^{i2\pi/d} \rangle$ the cyclic subgroup of $S^{1}_{\rm{std}}$ of order $d \geq 1$,
\item ${\bf D}_k^* = \pi^{-1}({\bf D})$ the \emph{binary dihedral} group (of order $4k$), 
\item ${\bf T}^* = \pi^{-1}({\bf T})$ the \emph{binary tetrahedral} group (order 24),
\item ${\bf O}^* = \pi^{-1}({\bf O})$ the \emph{binary octahedral} group (of order 48), and 
\item ${\bf I}^*= \pi^{-1}({\bf I})$ the \emph{binary icosahedral} group  (of order 120)
\end{itemize}
with the latter three groups being referred to collectively as the \emph{binary polyhedral} groups. Guided by Goursat's Lemma, one can find the finite groups of $\SO(4) = \rho(S^3 \times S^3)$ as follows.

For $n \geq 3$ odd and $k \geq 3$, we define the finite subgroup ${\bf D}_{2^kn}'$ of $\SO(4)$ as follows. Let $\Z_{2n}$ be the cyclic normal subgroup of ${\bf D}^*_{n}$ of order $2n$, we have an isomorphism $\phi: {\bf D}^*_{n}/\Z_{2n} \to \Z_{2^k}/\Z_{2^{k-1}}$. We then set 
$$\widetilde{{\bf D}}_{2^{k+1}n} = \{(x,y) \in {\bf D}^*_{n} \times {\bf C}_{2^k} : [y] = \phi([x]) \},$$
where $[x] \in {\bf D}^*_{n}/\Z_{2n}$ and $[y] \in \Z_{2^k}/\Z_{2^{k-1}}$ denote the corresponding cosets of $x$ and $y$. Then, we define
$${\bf D}_{2^kn}' = \rho ( \widetilde{{\bf D}}_{2^{k+1}n}) = \widetilde{{\bf D}}_{2^{k+1}n}/ \{(1,1), (-1,-1) \}.$$

For $k \geq 1$ we define the group ${\bf T}_{3^k8}'$ of $\SO(4)$ of order $3^k 8$ as follows. Let $H_8$ be the inverse image under $\pi: S^3 \to \SO(3)$ of the subgroup of ${\bf T}$ generated by $P$ and $Z$. Then $H_8$ is a normal subgroup of ${\bf T}^*$ and is isomorphic to $D_2^*$. Then, ${\bf C}_{3^k}/{\bf C}_{3^{k-1}}$ and ${\bf T}^*/ H_8$ are isomorphic to $\Z_{3}$. Therefore, there are precisely two isomorphisms $\Psi_1, \Psi_2: {\bf T}^*/ H_8 \to {\bf C}_{3^k}/{\bf C}_{3^{k-1}}$. Let $\Psi$ be one such isomorphism and define the following subgroup of $S^3 \times S^3$
$$\widetilde{{\bf T}}_{3^k 8}(\Psi) = \{ (x,y) \in {\bf T}^* \times {\bf C}_{3^k} : [y] = \Psi([x]) \}.$$
One can check that $\widetilde{{\bf T}}_{3^k 8}(\Psi_1)$ and $\widetilde{{\bf T}}_{3^k 8}(\Psi_2)$ are isomorphic, so we let $\widetilde{{\bf T}}_{3^k 8}$ denote any subgroup of $S^3 \times S^3$ isomorphic to $\widetilde{{\bf T}}_{3^k 8}(\Psi)$. Since $(-1,-1) \not\in \widetilde{{\bf T}}_{3^k 8}$, we see $\rho$ restricted to $\widetilde{{\bf T}}_{3^k 8}$ is injective and we define the following subgroup of $\SO(4)$
$${\bf T}_{3^k 8}' \equiv \rho( \widetilde{{\bf T}}_{3^k 8}).$$
We note that, when $k =1$, this group is isomorphic to ${\bf T}^*$.

For integers $q \geq 1$ and $p_1, p_2 \in \Z$ with $(q, p_1) = (q, p_2) =1$, let $\gamma_{q ; p_1,p_2} = R(\frac{2\pi p_1}{q}) \oplus R(\frac{2\pi p_2}{q}) \in SO(4)$ and $\Gamma_{q; p_1, p_2} = \langle \gamma_{q; p_1, p_2} \rangle \, \leq \rho(S^1_{\rm{std}} \times S^1_{\rm{std}})$. Setting $\widetilde{\gamma}_{q; p_1, p_2} = ( \exp (\frac{i\pi}{q}(p_1+ p_2)), \exp(\frac{i\pi}{q} (p_2 - p_1)) )$ we see that $\gamma_{q; p_1, p_2} = \rho(\pm \widetilde{\gamma}_{q; p_1, p_2} )$ and  define $\widetilde{\Gamma}_{q; p_1, p_2} = \langle \pm \gamma_{q; p_1, p_2} \rangle$. Then, every cyclic subgroup of $\SO(4)$ is conjugate via an element of $\SO(4)$ to $\Gamma_{q; p_1, p_2}$ for some choice of $p_1, p_2$ and $q$. Furthermore, choosing $k$ to be congruent to $p_1^{-1}$ modulo $q$, we see that $\Gamma_{q;p_1, p_2} = \Gamma_{q; 1, kp_2}$ and conclude that up to conjugation every cyclic subgroup of $\SO(4)$ is of the form $\Gamma_{q; 1, p}$, where $(q, p) = 1$.

\begin{dfn}\label{dfn:GroupTypes}
A finite subgroup $\Gamma$ of $\SO(4)$ is said to be of
\begin{enumerate}
\item \emph{Type I} if $\Gamma$ is isomorphic to $\Z_q$ for some $q \geq 1$, in which case $\Gamma$ is conjugate in $\SO(4)$ to $\Gamma_{q;1,p} = \rho(\widetilde{\Gamma}_{q;1,p})$ for some choice of $p$ relatively prime to $q$;
\item \emph{Type II} if $\Gamma$ is isomorphic to ${\bf D}_{2^kn}' \times \, {\bf C}_q \simeq \rho (\widetilde{{\bf D}}_{2^{k+1}n} \times \, {\bf C}_{q}) \leq \rho({\bf D}_n^{*} \times \, {\bf C}_{2^{k}q})$
for $(q, 2n)= 1$, $n \geq 3$ odd, $k \geq 3$ and whereby $\widetilde{{\bf D}}_{2^{k+1}n} \times \, {\bf C}_{q}$ we mean the subgroup of ${\bf D}_n^{*} \times {\bf C}_{2^{k}q}$ ($\leq S^3 \times S^{1}_{\rm{std}}$) generated by $\widetilde{{\bf D}}_{2^{k+1}n}$ and  $1\times \, {\bf C}_{q}$;
\item \emph{Type III} if $\Gamma$ is isomorphic to ${\bf D}_n^* \times {\bf C}_q $ for $(q,2n) = 1$ and $n \geq 2$;
\item \emph{Type IV} if $\Gamma$ is isomorphic to $
{\bf T}_{3^k 8}' \times {\bf C}_q = \rho(\widetilde{{\bf T}}_{3^{k}8} \times {\bf C}_{q}) \leq \rho({\bf T}^* \times {\bf C}_{q3^k})$ for $(q,6) = 1$ and $k \geq 1$; whereby $\widetilde{{\bf T}}_{3^{k}8} \times {\bf C}_{q}$ we mean the subgroup of ${\bf T}^* \times {\bf C}_{3^{k}q}$ ($\leq S^3 \times S^{1}_{\rm{std}}$) generated by $\widetilde{{\bf T}}_{3^{k}8}$ and $1 \times {\bf C}_{q}$;

\item \emph{Type V} if $\Gamma$ is isomorphic to ${\bf O}^* \times {\bf C}_q $ for $(q,6) = 1$;
\item \emph{Type VI} if $\Gamma$ is isomorphic to ${\bf I}^* \times {\bf C}_q $ for $(q,30) = 1$.
\end{enumerate}

\end{dfn}

\begin{rem}
In general, up to isomorphism, groups of Type II-VI are of the form $\Delta \times {\bf C}_q = \rho(\widetilde{\Delta} \times {\bf C}_q)$, where $(1)$ 
$\Delta$ is ${\bf D}_{2^{k}n}'$, ${\bf T}_{3^{k}8}'$, ${\bf O}^*$, ${\bf I}^*$, or ${\bf D}_n^{*}$ and $(2)$ $(q, |\Delta|) =1$.  
\end{rem}

\begin{rem}\label{rem:BinarySubgroups} The binary dihedral subgroups of $\SO(4)$ are the Type III groups with $q =1$. Recall that ${\bf T}_{3^k 8}'$ is isomorphic to ${\bf T}^*$ when $k=1$, hence the binary polyhedral subgroups of $\SO(4)$ correspond to the groups of Type V and VI with $q =1$ and Type IV with $q=k=1$. 
\end{rem}

\begin{rem} The manifold ${\bf I}^* \backslash S^3$ is known as the \emph{Poincare dodecahedral space}. It is the unique homology three-sphere with finite non-trivial fundamental group. It has been proposed as an explanation for weak wide-angle temperature cancellations in the cosmic microwave background \cite{Weeks}.  
\end{rem}

\begin{dfn}
An elliptic three-manifold is said to be a \emph{lens space} if $M$ is diffeomorphic to $L(q;1,p) \equiv \Gamma_{q;1,p} \backslash S^3$ for some choice of relatively prime integers $q \geq 1$ and $p$.
\end{dfn}

The following theorem due to Seifert and Threlfall provides a classification of all elliptic three-manifolds. In light of the Elliptization Theorem, it also provides a topological categorization of all closed three-manifolds with finite fundamental group.

\begin{thm}[cf. \cite{SeifertThrelfall, Cohen} and Proposition 1.1 of \cite{Ikeda79}] \label{SeifertThrelfalltheorem}
Every finite subgroup of $\SO(4) = \rho(S^3 \times S^3)$ is of Type I-VI. If $\Gamma_1, \Gamma_2 \leq \SO(4)$ act freely on $S^3$ and are not cyclic of order at least three, then $\Gamma_1$ and $\Gamma_2$ are isomorphic if and only if they are conjugate in $\operatorname{O}(4)$. More specifically, for $\Gamma_1, \Gamma_2 \leq \SO(4)$ acting freely on $S^3$, we have the following.

\begin{enumerate}
\item For $j =1,2$, let $\Gamma_j \leq \SO(4)$ be Type I and of order 1 or 2. Then, $\Gamma_1$ and $\Gamma_2$ are isomorphic if and only if $\Gamma_1$ and $\Gamma_2$ are conjugate in $\operatorname{O}(4)$. In this case we conclude both groups are isomorphic to the trivial group or $ \langle -I_4 \rangle$, where $I_4$ is the identity map on $\R^4$. 

\item For $j =1,2$, let $\Gamma_j \leq \SO(4)$ be Type II-VI. Then, $\Gamma_1$ and $\Gamma_2$ are isomorphic if and only if $\Gamma_1$ and $\Gamma_2$ are conjugate in $\operatorname{O}(4)$. In which case we see the Type II-VI subgroups are of the form $\rho((\widetilde{\Delta} \times {\bf C}_q)^{(g_1,g_2)})$ or $\rho(({\bf C}_q \times \widetilde{\Delta})^{(g_1,g_2)})$, where $\widetilde{\Delta} = {\bf D}_k^*, {\bf T}^*, {\bf I}^*, {\bf O}^*, \widetilde{{\bf D}}_{2^{k+1}n}$ or  $\widetilde{{\bf T}}_{3^{k}8}$, $(g_1, g_2) \in S^3 \times S^3$ and by $\widetilde{\Delta} \times {\bf C}_q$ (respectively, ${\bf C}_q \times \widetilde{\Delta}$) we mean the subgroup of $S^3\times S^{1}_{\rm{std}}$ (respectively, $S^{1}_{\rm{std}} \times S^3$) generated by $\widetilde{\Delta}$ and  $1\times {\bf C}_q$ (respectively, $\widetilde{\Delta}$ and  ${\bf C}_q \times 1$).

\item For $j =1,2$, let $\Gamma_j \leq \SO(4)$ be Type I of order at least $q \geq 3$. Then, up to conjugation in $\SO(4)$, $\Gamma_j = \Gamma_{q; 1,p_j}$. Then, $\Gamma_1 \simeq \Gamma_2$ if and only if $p_2 \equiv \pm p_1, \pm p_1^{-1} \mod q$, which is equivalent to $\Gamma_1$ and $\Gamma_2$ are conjugate in $\operatorname{O}(4)$. In fact, $p_2 \equiv \pm p_1^{-1} \mod q$ if and only if $\Gamma_1$ and $\Gamma_2$ are conjugate in $\SO(4)$, and $p_2 \equiv \pm p_1 \mod q$ if and only if $\Gamma_1$ and $\Gamma_2$ are conjugate in $\operatorname{O}(4) - \SO(4)$.
\end{enumerate}
\end{thm}

\begin{rem}\label{rem:LeftRight}
It will be useful to recall that $L = \rho(S^3 \times 1)$ and $R = \rho(1 \times S^3)$ are normal subgroups of $\SO(4)$ and $L^A = R$ for any $A \in \operatorname{O}(4) - \SO(4)$. 
\end{rem}

\begin{rem}\label{rem:TypeIConjugacy}
Explicit diffeomorphisms between the lens spaces $L(q;1, p)$, $L(q;1, -p)$, $L(q;1, p^{-1})$ and $L(q;1, - p^{-1})$ are found by noting that the group $\Gamma_{q;1,p}$ is conjugate to $\Gamma_{q;1, p^{-1}}$ via 
$$F_1 \equiv  \left(\begin{array}{cc}0 & I_2 \\I_2 & \ 0\end{array}\right)\in \SO(4),$$ where $I_2$ denotes the $2\times 2$ identity matrix,  and $\Gamma_{q;1,-p} = \Gamma_{q;1,p}^{F_2}$, where 
$$F_2 \equiv I_2 \oplus \left(\begin{array}{cc}0 & -1 \\-1 & 0\end{array}\right) \in \operatorname{O}(4) - \SO(4).$$
In particular, since $F_1 \in N_{\Diff(S^3)}(L)$, Lemma~\ref{lem:Basic} informs us that  $\overline{\mathscr{R}}_{S^3}^{\Gamma_{q;1,p}}(L(q;1,p))$ and $\overline{\mathscr{R}}_{S^3}^{\Gamma_{q;1,p^{-1}}}(L(q;1,p^{-1}))$ determine the same isometry classes of locally homogeneous metrics on the lens space $L(q;1,p)$. 
\end{rem}

Theorem \ref{SeifertThrelfalltheorem}  implies the following classification of spherical space forms of constant positive sectional curvature. 

\begin{cor}
Let $\mathbb{S}^3$ denote the three-sphere equipped with a metric of constant sectional curvature one.
Suppose $\Gamma_1 \backslash S^3$ and $\Gamma_2 \backslash S^3$ are two elliptic manifolds, where $\Gamma_1$ and  $\Gamma_2$ are isomorphic.
\begin{enumerate} 
\item If $\Gamma_1 \simeq \Gamma_2$ is Type I of order at most two or $\Gamma_1 \simeq \Gamma_2$ is Type II -VI, then $\Gamma_1$ and $\Gamma_2$ are conjugate in $\operatorname{O}(4)$ and the manifolds $\Gamma_1 \backslash \mathbb{S}^3$ and $\Gamma_2 \backslash \mathbb{S}^3$ are isometric.
\item If $\Gamma_1 \simeq \Gamma_2$ is Type I with $q \geq 3$ (i.e., cyclic of order at least three), 
then for appropriate choices of $p_1$ and $p_2$ we have $\Gamma_j \backslash S^3 \simeq L(q; 1, p_j)$ for $j =1, 2$, and the following four statements are equivalent
\begin{enumerate}
\item $\Gamma_1 \backslash S^3$ and $\Gamma_2 \backslash S^3$ are diffeomorphic
\item $p_2 \equiv \pm p_1 \mod q$ or $p_2 \equiv \pm p_1^{-1} \mod q$
\item $\Gamma_1$ and $\Gamma_2$ are conjugate in $\operatorname{O}(4)$
\item $\Gamma_1 \backslash \mathbb{S}^3$ and $\Gamma_2 \backslash \mathbb{S}^3$ are isometric
\end{enumerate}
\end{enumerate}
\end{cor}

%%%%%%%%%%%%%%%%%%%%%%%%
\subsection{Isometry classes of left-invariant metrics on $S^3$}\label{sec:IsometryClassesOfLeftInvariantMetrics}
Let $G$ be a $d$-dimensional compact Lie group endowed with a fixed bi-invariant metric $g_0$ and let $\BiInner$ denote its value at the identity. The collection of left-invariant metrics on $G$ will be denoted by $\mathscr{R}_{\rm{left}}(G)$ and the corresponding collection of isometry classes of left-invariant metrics on $G$ will be denoted by $\overline{\mathscr{R}}_{\rm{left}}(G)$. An arbitrary left-invariant metric $g \textcolor{red}{}$ is completely prescribed by its value $\Inner$ at the identity and is related to our background bi-invariant metric $g_0$ via $$\langle X , Y \rangle = \langle \Omega (X) , Y \rangle_0,$$ where $\Omega: \germ{g} \to \germ{g}$ is a positive definite $\BiInner$-self-adjoint linear transformation. The eigenspaces, eigenvalues $\eta_1^2, \ldots, \eta_d^2$ and corresponding eigenvectors $u_1, \ldots, u_d$ of $\Omega$ are referred to as the \emph{metric eigenspaces\label{page:MetricEigenvalues}}, \emph{metric eigenvalues}, and \emph{metric eigenvectors} of $g$ (with respect to $g_0$), respectively. When all the metric eigenspaces of $g$ are one-dimensional, we will say that $g$ is \emph{generic}. 
Finally, mirroring the notation used in Section~\ref{sec:Notation}, given $\Gamma \leq \operatorname{Diff}(G)$ acting freely and properly discontinuously on $G$, we let $\mathscr{R}_{\rm{left}}^{\Gamma}(G)$ be the collection of $\Gamma$-invariant metrics in $\mathscr{R}_{\rm{left}}(G)$. Then, any  $h \in \mathscr{R}_{\rm{left}}^{\Gamma}(G)$ induces a locally homogeneous metric $h_\Gamma$. In this case we write:
\begin{itemize}
\item $\mathscr{R}_{\rm{left}}^{\Gamma}(\Gamma \backslash G) = \{h_\Gamma : h \in \mathscr{R}_{\rm{left}}^{\Gamma}(G) \}$ to denote the collection of locally homogeneous metrics on $\Gamma \backslash G$ that are the quotient metrics associated to metrics on $G$ that are simultaneously left-invariant and $\Gamma$-invariant, 
\item $\overline{\mathscr{R}}_{\rm{left}}^{\Gamma}(G) = \{ [h_\Gamma] : h \in \mathscr{R}_{\rm{left}}^{\Gamma}(G) \}$ to denote the isometry classes of the preceding metrics,
\item $\overline{\mathscr{R}}_{\rm{left}}( \Gamma \backslash G)$ denotes the isometry classes of metrics on $\Gamma \backslash G$ covered by a left-invariant metric on $G$.
\end{itemize}

\begin{dfn}\label{dfn:Multiset}
A \emph{multiset} in $\R$ is a map $m: \R \to \N \cup \{0\}$, where we think of $m(x)$ as the multiplicity of $x$ in the multiset. A multiset $m$ is said to be a \emph{$k$-multiset}, for $k \in \N$, if $m$ is non-zero at finitely many distinct values $x_1, \ldots , x_q$ and $\sum m(x_j) = k$. A $k$-multiset $m$ will be denoted by $[x_{11}, \ldots x_{1m(x_1)}, \ldots , x_{q1} \ldots , x_{qm(x_q)} ]$, where $x_{ij} = x_i$ for $j = 1, \ldots , m(x_i)$. The collection of $k$-multisets consisting of positive numbers will be denoted by $\mathscr{M}^{+}_k$. 
\end{dfn}

\noindent 
Define $\mathcal{E} : \mathscr{R}_{\rm{left}}(G) \to \mathscr{M}^{+}_{n}$ to be the map which sends a left-invariant metric $g \in \mathscr{R}_{\rm{left}}(G)$ to the $n$-multiset $[x_1 = x_1(g), \ldots , x_n = x_n(g) ]$ consisting of its metric eigenvalues (with respect to $g_0$).

In the case where $G$ is $S^3$ (respectively, $\SO(3)$) it has been shown that the metric eigenvalues determine a left-invariant metric up to isometry.  

\begin{prop}[\cite{Tapp}]\label{prop:MetricEigenvaluesSO(3)}
Two left-invariant metrics on $\SO(3)$ have the same multi-set of eigenvalues (with respect to a background bi-invariant metric $g_0$) if and only if they are isometric. That is, the map $\mathcal{E}$ induces a bijection $\overline{\mathcal{E}}: \overline{\mathscr{R}}_{\rm{left}}(\SO(3)) \to \mathscr{M}^{+}_{3}$ between the isometry classes of left-invariant metrics on $\SO(3)$ and the set of $3$-multisets of positive numbers.
\end{prop}

When $\Gamma \leq \Diff(G)$ acts properly discontinuously and $g \in \mathscr{R}_G^{\Gamma}(G)$, the metric eigenvalues of $g_\Gamma$ are defined to be the metric eigenvalues of $g$. The preceding Proposition implies that for any $\Gamma \leq \operatorname{SO}(4)$ of Type I-VI the map $\mathcal{E} : \overline{\mathscr{R}}_{\rm{left}} (\Gamma \backslash S^3) \to \mathscr{M}^{+}_{3}$ sending an isometry class of a locally left-invariant metric on $\Gamma \backslash S^3$ to the eigenvalues of a left-invariant covering metric is well-defined and the following is immediate.

\begin{cor}\label{cor:MetricEigenvaluesSO(3)}
Let $\Gamma \leq \operatorname{SO}(4)$ be a subgroup of Type I-VI that acts freely on $S^3$. The map $\overline{\mathcal{E}} : \overline{\mathscr{R}}_{\rm{left}}(\Gamma \backslash S^3) \to \mathscr{M}_3^{+}$ given by $[h] \mapsto [x = x(h), y = y(h), z= z(h) ]$ is well-defined and $\overline{\mathcal{E}}([h_1]) = \overline{\mathcal{E}}([h_2])$ implies $[h_1]$ and $[h_2]$ are locally isometric. In the case where $\Gamma$ is trivial or $\Z_2$, this map is bijective; that is, the left-invariant metrics on $\SO(3)$ (respectively $S^3$) are determined up to isometry by the multiset of their metric eigenvalues.
\end{cor}

\begin{proof} 
The general statement follows from Proposition~\ref{prop:MetricEigenvaluesSO(3)} and the fact that every locally homogeneous metric on an elliptic three-manifold is locally isometric to a left-invariant metric on $S^3$. The case where $\Gamma\backslash S^3 = \SO(3)$ (i.e., $\Gamma \simeq \Z_2$) is precisely Proposition~\ref{prop:MetricEigenvaluesSO(3)}. And, the case where $\Gamma \backslash S^3 = S^3$ (i.e., $\Gamma$ is trivial) follows from Proposition~\ref{prop:MetricEigenvaluesSO(3)} since the two-fold covering map  $\pi: S^3 \simeq \SU(2) \rightarrow \SO(3) = \SU(2) / Z(\SU(2))$ induces a bijection between isometry classes of left-invariant metrics on $S^3 \simeq \SU(2)$ and $\SO(3)$.
\end{proof}

In subsequent sections it will be useful to distinguish the so-called Berger metrics on $S^3$ (respectively, $\SO(3)$) which are left-invariant metrics formed by scaling the standard constant curvature metric along the fibers of the Hopf fibration.

\begin{dfn}\label{dfn:BergerMetrics}
A left-invariant metric on $S^3$ (respectively, $\SO(3)$) is said to be \emph{naturally reductive} or \emph{Berger}, if it has a metric eigenvalue of multiplicity at least two. 
\end{dfn}

%%%%%%%%%%%%%%%%%%%%%%%%%%%%%%%%%%
\subsection{Isometry groups of left-invariant metrics on $S^3$}\label{sec:IsometryGroups}
In preparation for our proof of Theorem~\ref{thm:LocHomogEllipticThreeManifolds}, it will be useful to know which subgroups $\Gamma \leq \SO(4)$ of Type I-VI can appear as subgroups of $\Isom(S^3, g)$, for a given left-invariant metric $g$ on $S^3$ (see Corollary~\ref{cor:DiscreteSubgroupsIsometryGroupsS^3}). We will achieve this by computing the isometry groups of the left-invariant metrics on $S^3$ (see Proposition~\ref{prop:IsometryGroupsS^3}). We begin with a general discussion of isometry groups of left-invariant metrics on a Lie group. 

\begin{lem}\label{lem:AutoIsometry}
Let $G$ be a Lie group with automorphism group $\operatorname{Aut}(G)$. Then, for any $g \in \mathscr{R}_{\rm{left}}(G)$ and $\alpha \in \operatorname{Aut}(G)$, $\alpha$ is an isometry of $(G,g)$  if and only if $\Inner \equiv g_e$ is $\alpha_{*}$-invariant.
\end{lem}

\begin{proof}
The ``only if'' statement is clear. Now, suppose $\Inner$ is $\alpha_{*}$-invariant. Fix a $\Inner$-orthonormal basis $\{e_1, \ldots , e_n \}$ of $T_eG$ and let $E_1, \ldots , E_n$ be its extension to a left-invariant $g$-orthonormal framing of $G$. Then, since $\alpha_{*}(E_{jp}) = \alpha_{*}(L_{p*}(e_j)) = L_{\alpha(p)*}(\alpha_{*}(e_j))$, we find 
\begin{eqnarray*}
g(\alpha_{*}(E_{ip}), E_{jp}) &=& g(L_{\alpha(p)*}(\alpha_{*}(e_i)), L_{\alpha(p)*}(\alpha_{*}(e_j))) \\
&=& \langle \alpha_{*}(e_i), \alpha_{*}(e_j) \rangle \\
&=& \langle e_j , e_i \rangle \\
&=& \delta_{ij} \\
&=& g(E_{ip}, E_{jp}).
\end{eqnarray*}
Hence, $\alpha$ is an isometry of $(G,g)$.
\end{proof}

The following statement is an immediate consequence.

\begin{cor}\label{cor:AutoIsometry}
Let $G$ be a Lie group with left-invariant metric $g$ and for any $h \in G$, let $C_h \equiv L_{h} \circ R_{h^{-1}}$ be conjugation by $h$. Given $h \in G$ we see the following statements are equivalent.
\begin{enumerate}
\item $R_h \in \Isom(G,g)$
\item $C_{h} \in \Isom(G,g)$
\item $\Inner = g_e$ is $\Ad(h)$-invariant.
\end{enumerate}
\end{cor}

We recall the following facts about the isometry groups of compact simple Lie groups. 

\begin{thm}\label{thm:IsometryGroupofSimpleLieGroup}[ see Thm. 4 and 6 of \cite{Onishchik}]
Let $\cdot$ be the group operation given by composition of functions and let $G$ be a compact simple Lie group equipped with a left-invariant metric $g$. 
\begin{enumerate}
\item If $g$ is not bi-invariant, then $\Isom(G,g) = L(G) \cdot B$, where $B$ is a subgroup of  $\operatorname{Aut}(G)$.
\item If $g$ is bi-invariant, then $\Isom(G,g) = \left( L(g) \cdot \operatorname{Aut}(G) \right) \rtimes \langle \iota \rangle$, where $\iota : G \to G$ is the inversion map.
\end{enumerate}
\end{thm}

\begin{rem}
The above follows from more results of Onischick that generalize \cite{OchiaiTakahashi} and \cite[p. 24]{DAtriZiller} to homogeneous spaces realized as quotients of simple Lie groups. Please note that the statement found in \cite[p. 24]{DAtriZiller}, is only valid for left-invariant metrics that are \emph{not} bi-invariant. 
\end{rem}

We now introduce the following lemma concerning a joint isometry of inner-product spaces. When considering a compact simple Lie group equipped with a left-invariant metric that is not bi-invariant, it will allow us to pinpoint the subgroup $B$ identified in the previous theorem, thus computing the full isometry group. 

\begin{lem}\label{lem:JointIsometries}
Let $\Inner_0$ and $\Inner$ be inner products on a finite-dimensional vectror space $V$ with orthogonal groups $\operatorname{O}(V, \Inner_0)$ and $\operatorname{O}(V, \Inner)$, respectively. Now, let $\Omega: (V, \Inner_0) \to (V, \Inner_0)$ be the unique self-adjoint map such that $\langle \cdot , \cdot \rangle = \langle \Omega(\cdot), \cdot \rangle_0$ and let $V = \oplus_{j=1}^{m} V_{\lambda_j}$ be the $\Inner_0$-orthogonal decomposition of $V$ into $\Omega$-eigenspaces. If $T \in \operatorname{O}(V, \Inner_0)$, then $T \in \operatorname{O}(V, \Inner)$ if and only if $T$ preserves the $\Omega$-eigenspaces.
\end{lem}

\begin{proof}
We begin by assuming $T$ is an isometry of the inner-product space $(V, \Inner_0)$ that preserves the $\Omega$-eigenspaces. Now, suppose $v_j \in V_{\lambda_j}$ and $v_{k} \in V_{\lambda_k}$ are $\Omega$-eigenvectors. Then, 
\begin{eqnarray*}
\langle T(v_j), T(v_k) \rangle &=& \langle \Omega(T(v_j)), T(v_k) \rangle_0 \\
&=& \langle \lambda_{j} T(v_j), T(v_k) \rangle_0 \\
&=& \langle T(\lambda_j v_j) , T(v_k) \rangle_0 \\
&=& \langle \lambda_j v_j, v_k \rangle_0 \\
&=& \langle \Omega(v_j), v_k \rangle_0\\
&=& \langle v_j, v_k \rangle.
\end{eqnarray*}
It follows that $T$ is an isometry of $(V, \Inner)$.

Now, assume $T$ is simultaneously an isometry of $(V, \Inner_0)$ and $(V, \Inner)$. Also, let $v_j \in V_{\lambda_j}$ and $v_k \in V_{\lambda_k}$ be $\Omega$-eigenvectors. Then, on the one hand, we have 
$$\langle v_j , T(v_k) \rangle = \langle \Omega(v_j), T(v_k) \rangle_0 = \lambda_j \langle v_j, T(v_k) \rangle_0,$$
while
\begin{eqnarray*}
\langle v_j, T(v_k) \rangle &=& \langle T^{-1}(v_j), v_k \rangle \\
&=& \langle T^{-1}(v_j), \Omega (v_k) \rangle_0 \\
&=& \lambda_k \langle T^{-1}(v_j), v_k \rangle_0\\
&=& \lambda_k \langle v_j, T(v_k) \rangle_0.
\end{eqnarray*}  
Therefore, $(\lambda_j - \lambda_k)\langle v_j, T(v_k) \rangle_0 = 0$ and, in the event that $j \neq k$, we conclude that $\langle v_j, T(v_k) \rangle_0 =0$. Consequently, $T$ is an isometry of $(V, \Inner_0)$ that preserves the $\Omega$-eigenspaces.
\end{proof}

\begin{thm}\label{thm:IsometryGroupofSimpleLieGroup2}\label{thm:IsometryGroupSimple}
Let $G$ be a compact simple Lie group. Fix a background bi-invariant metric $g_0$ on $G$ and let $\Inner_0$ denote its value at the identity element. Now, consider a left-invariant metric $g$ on $G$ and set $\Inner$ equal to the value of $g$ at the identity element. Then, the isometry group of $(G,g)$ is

$$\Isom(G,g) = \left\{\begin{array}{ll}
(L(G)\cdot \operatorname{Aut}(G)) \rtimes \langle \iota \rangle, & \mbox{if $g$ is bi-invariant,}\\
L(G) \cdot \operatorname{Aut}^{g}(G), & \mbox{if $g$ is not bi-invariant,}
\end{array}
\right.$$
where $\iota: G \to G$ is inversion and 
$\operatorname{Aut}^{g}(G)$ is the subgroup of $\operatorname{Aut}(G)$ consisting of $\alpha$ with the property that $\alpha_{*} \in \operatorname{O}(T_eG, \Inner_0)$ and $\alpha_{*}$ preserves the metric eigenspaces of $g$ with respect to $g_0$.
\end{thm}

\begin{proof}
This is an immediate consequence of Theorem~\ref{thm:IsometryGroupofSimpleLieGroup} and Lemma~\ref{lem:JointIsometries}.
\end{proof}

We now specialize to the isometry groups of left-invariant metrics on $S^3$. In the event that $g$ is a bi-invariant metric on $S^3$, it is well-known that $\Isom(S^3,g)$ is isomorphic to $\operatorname{O}(4)$, the orthogonal group on four-dimensional Euclidean space. For a left-invariant metric on $S^3$ that is not bi-invariant, Theorem~\ref{thm:IsometryGroupofSimpleLieGroup}(2) implies $\Isom(S^3, g) \leq \rho(S^3 \times S^3) = \SO(4)$. Therefore, the isometry group of a left-invariant metric on $S^3$ contains orientation-reversing isometries if and only if the metric is bi-invariant. Now, since there are no outer automorphisms of $S^3$, the preceding theorem informs us that, in order to identify the isometry group of a left-invariant metric $g$ on $S^3$ that is generic or Berger of non-constant sectional curvature, we need to identify the subgroup of $S^3$ consisting of the elements $h$ such that $\Ad(h)$ preserves the metric eigenspaces of $g$ with respect to some background bi-invariant metric $g_0$.

\begin{nota}
For the remainder of this subsection, let $g_0$ denote the bi-invariant metric on $S^3$ induced by the Killing form and let $\Inner_0$ denote its value at $1 \in S^3$. And, for an arbitrary left-invariant metric $g$ on $S^3$ we let $\Inner$ denote its value at $1$.
\end{nota}

\begin{nota}\label{lem:OrientationReversing}
Since $h \in S^3 \mapsto \Ad(h) \in \SO(T_1S^3, \Inner_0)$ is a double cover, given a two-dimensional subspace $W \leq T_{1}S^3$ (up to multiplication by $\pm 1$) there is a unique element $q_{W} \in S^3 \backslash \{\pm 1\}$ such that $\Ad(q_W) \upharpoonright W = - I_2$ and $\Ad(q_W) \upharpoonright W^\perp$ is the identity, where $W^\perp$ is the $\Inner_0$-orthogonal complement of $W$.
\end{nota}

\begin{rem}\label{nota:qWIsometries}
The element $q_W$ defined above is an element of order four in $S^3$ with $q_W^2 = -1$ and $q_W^3 = -q_W$. Indeed, by the definition of $q_W$ we see $\Ad(q_W^2)$ is the identity. Therefore, $q_W^2$ is in the kernel of the map $x \in S^3 \mapsto C_x \in \SO(3)$, which is equal to $\{ \pm 1\}$. Since, $q_W \neq \pm 1$, we conclude $q_W^2 = -1$. 
\end{rem}
 
\begin{cor}\label{cor:qWIsometries}
Let $g \in \mathscr{R}_{\rm{left}}(S^3)$ and $W \leq T_1S^3$ be a two-dimensional subspace spanned by two orthogonal metric eigenvectors of $g$ (with respect to $g_0$). Then, $R_{q_W}, R_{-q_W} \in \Isom(S^3, g)$. In particular, we have
$$\rho(S^3 \times \{\pm q_W\}) \leq \Isom(S^3, g).$$
\end{cor}

\begin{proof}
This is an application of Corollary~\ref{cor:AutoIsometry} to the present situation.
\end{proof}

We will now compute the isometry groups of left-invariant metrics on $S^3$ that are not bi-invariant. 

\begin{comp}[The isometry group of a generic left-invariant metric on $S^3$]\label{comp:GenericMetrics}
Let $g \in \mathscr{R}_{\rm{left}}(S^3)$ be a generic left-invariant metric with pairwise distinct metric eigenvalues $\eta_1^2$, $\eta_2^2$ and $\eta_3^2$ (with respect to $g_0$) and corresponding metric eigenvectors $u_1$, $u_2$ and $u_3$. Set 
$$W_1(g) = \Span \{u_2, u_3\}, W_2(g) = \Span \{u_1, u_3\} \mbox{ and } W_3(g) = \Span \{u_1, u_2\}.$$
Direct inspection reveals that $\{\pm 1\} \cup \left( \cup_{j=1}^{3} \{\pm q_{W_j(g)} \} \right)$ is the subgroup of $S^3$ consisting of precisely those elements $h \in S^3$ such that $\Ad(h)$ preserves the metric eigenspaces of $g$ with respect to $g_0$. Consequently, by Theorem~\ref{thm:IsometryGroupofSimpleLieGroup2} and the fact that all automorphisms of $S^3$ are inner, we find $\Isom(S^3,g) = \rho(S^3 \times \{\pm 1\}) \cup \left( \cup_{j=1}^{3} \rho(S^3 \times \{ \pm q_{W_j(g)} \}) \right)$.
\end{comp}

\begin{comp}[The isometry group of a Berger metric of non-constant sectional curvature]\label{comp:NCCBergerMetrics}
Let $g \in \mathscr{R}_{\rm{left}}(S^3)$ be a Berger metric of non-constant sectional curvature with metric eigenvalues $\eta_1^2 = \eta_2^2 \neq \eta_3^2$ (with respect to $g_0$) and corresponding metric eigenvectors $u_1$, $u_2$ and $u_3$. Further, let $W = \Span \{u_1, u_2 \}$ be so that $W^\perp = \Span \{u_3\}$ is the $\Inner$-orthogonal (and, hence, $\Inner_0$-orthogonal) complement of $W$.

\begin{enumerate}
\item By $S^1(g)$ we will denote the unique one-parameter subgroup of $S^3$ so that the special orthogonal group of $(W, g_e \upharpoonright W)$ is 
$$\SO(W, g_e \upharpoonright W) = \{ f_{*} : f \in \rho(\Delta S^1(g)) \} = \{\Ad(k) : k \in S^{1}(g) \}$$
and for any $k \in S^{1}(g)$, $\Ad(k) \upharpoonright W^\perp$ is the identity.
\item Let $q_g \in S^3$ be such that 
$$\Ad(q_g)\upharpoonright W \in \operatorname{O}(W, g_e \upharpoonright W) \backslash \SO(W, g_e \upharpoonright W)$$
and $\Ad(q_g)\upharpoonright W^\perp$ is the negative of the identity.
\end{enumerate}
(Please see Example~\ref{exa:Berger} below for a concrete instance of these constructions.) Now, direct inspection shows that $S^1(g) \cup q_gS^1(g)$ is the subgroup of $S^3$ consisting of precisely those elements $h \in S^3$ for which $\Ad(h)$ preserves the metric eigenspaces of $g$ (with respect to $g_0$). Therefore, by Theorem~\ref{thm:IsometryGroupofSimpleLieGroup2} and the fact that all automorphisms of $S^3$ are inner, we conclude $\Isom(S^3, g) = \rho(S^3 \times S^1(g)) \cup \rho(S^3 \times q_gS^1(g))$.
\end{comp}

\begin{rem}
Clearly, the element $g_q$ is not unique and, by the reasoning in Remark~\ref{nota:qWIsometries}, it is an order four element with $q_g^2 = -1$ and $q_g^3 = -q_g$. One can also check that $q_g$ normalizes $S^1(g)$ in $S^3$.
\end{rem}

\begin{exa}[Finding $S^1(g)$ and $q_g$ for a Berger metric]\label{exa:Berger}
Recall that $S^3 \subset \Span\{1,i,j,k \}$ is the space of unit length quaternions and $T_1S^3 = \Span\{i,j,k\}$. Now, let $g \in \mathscr{R}_{\rm{left}}(S^3)$ be a Berger metric of non-constant sectional curvature obtained through a non-trivial rescaling of $\Inner_0$ in the $i$-direction. Then, we have
\begin{itemize}
\item $S^1(g) = S^1_{\rm{std}} \equiv \{e^{i\theta} : \theta \in \R \}$, and
\item we may take $q_g = j$.
\end{itemize}
Indeed, $\Ad(j)(i) = -i$, $\Ad(j)(j) = j$ and $\Ad(j)(k) = -k$.
\end{exa}

We summarize the preceding discussion in the following proposition.

\begin{prop}\label{prop:IsometryGroupsS^3}\label{prop:IsometryGroupS^3}
The isometry group of a left-invariant metric on $S^3$ is given by

$$\Isom(S^3, g) = \left\{
\begin{array}{ll}
\operatorname{O}(4), & \mbox{if $g$ is bi-invariant},\\
\rho(S^3 \times S^1(g)) \cup \rho(S^3 \times q_g S^1(g)), &  \mbox{if $g$ is Berger of non-constant}\\ 
& \mbox{sectional curvature}, \\
\rho(S^3 \times \{\pm 1\}) \cup \left( \cup_{j=1}^{3} \rho(S^3 \times \{ \pm q_{W_j(g)} \}) \right), & \mbox{if $g$ is generic}.
\end{array}
\right.$$
In particular, $(S^3, g)$ admits orientation-reversing isometries if and only if $g$ is bi-invariant and the connected component of the identity of $\Isom(S^3,g)$ is given by 
$$\Isom(S^3, g)^o \simeq \left\{ \begin{array}{ll} 
\SO(4) = \rho(S^3 \times S^3), & \mbox{if $g$ is bi-invariant}, \\
\rho(S^3 \times S^1(g)), & \mbox{if $g$ is Berger of non-constant sectional curvature},\\
\rho(S^3 \times \{\pm 1\}), & \mbox{if $g$ has three pairwise distinct eigenvalues}.
\end{array} \right.$$
\end{prop}

The following corollary sheds light on which finite subgroups of $\SO(4)$ that act freely on $S^3$ can appear as the subgroup of the isometry group of a given left-invariant metric $g$ and will be useful in our proof of Theorem~\ref{thm:LocHomogEllipticThreeManifolds}.

\begin{cor}\label{cor:DiscreteSubgroupsIsometryGroupsS^3}
Let $g$ be a left-invariant metric on $S^3$ that is not bi-invariant and $\Gamma$ a finite subgroup of $\SO(4)$ that acts freely on $S^3$. If $\Gamma$ is Type II -VI and $\Gamma \leq \Isom(S^3, g)$, then (up to conjugation in $\SO(4)$) $\Gamma$ is the subgroup of $\rho(S^3 \times S^{1}_{\rm{std}})$ given by 
\begin{enumerate}
\item $\Gamma_{q;p_1, p_2}$ for $p_1$ and $p_2$ integers relatively prime to $q$, or  
\item $\rho(\widetilde{\Delta} \times {\bf C}_q)$, where $\widetilde{\Delta}$ is ${\bf D}_k^{*}$, ${\bf T}^*$, ${\bf I}^*$, ${\bf O}^*$, $\widetilde{{\bf D}}_{2^{2k+1}n}$ or $\widetilde{{\bf T}}_{3^{k}8}$ and $\widetilde{\Delta} \times {\bf C}_q$ denotes the subgroup of $S^3 \times S^{1}_{\rm{std}}$ generated by $\widetilde{\Delta}$ and $1 \times {\bf C}_q$.
\end{enumerate}
\end{cor}

\begin{proof}
Follows from Theorem~\ref{SeifertThrelfalltheorem} and Proposition~\ref{prop:IsometryGroupsS^3} by direct inspection.
\end{proof}

%\begin{rem}
%In fact, using Goursat's Lemma, we can conclude that, up to conjugation by an automorphism, a connected and closed subgroup of $S^3 \times S^3$ is trivial, $S^1 \times S^1$, $S^3 \times 1$, $S^3 \times S^1$, a graph of a surjective homomorphisms $\theta: S^1 \to S^1$ or the graph of an automorphism $\theta: S^3 \to S^3$.
%\end{rem}

%%%%%%%%%%%%%%%%%%%%%%%%
\subsection{Classification of locally homogeneous elliptic three-manifolds}\label{sec:LocHomogElliptic}
The goal of this subsection is to provide a proof of Theorem~\ref{thm:LocHomogEllipticThreeManifolds}, which classifies the locally homogeneous elliptic three-manifolds and to determine when such a space is homogeneous. Since the universal cover of a complete locally homogeneous Riemannian manifold is homogeneous \cite[p. 692]{Singer} and up to isometry the homogeneous metrics on $S^3$ are the left-invariant metrics, to classify the locally homogeneous metrics on an elliptic three-manifold, for each left-invariant metric $g$ on $S^3$ we will need to identify the finite subgroups $\Gamma \leq \Isom(S^3, g) \leq O(4)$ acting freely on $S^3$ (Corollary~\ref{cor:DiscreteSubgroupsIsometryGroupsS^3}) and determine when two locally homogeneous elliptic three-manifolds are isometric. Moreover, we will need a criterion (Proposition~\ref{prop:HomogeneousCriterion}) to determine when a locally homogeneous metric is homogeneous.

In \cite{LSS}, a complete locally homogeneous metric on a manifold $M$ is called a \emph{geometric structure}. Further, a geometric structure on $M$ is said to be \emph{modeled} on the geometry $(X,G)$ or is said to be an \emph{$(X,G)$-geometric structure} if its universal Riemannian cover is isometric to $X$ equipped with a $G$-invariant metric. In the event that $(X,G)$ is a maximal geometry, the $(X,G)$-geometric structures on $M$ are called \emph{maximal geometric structures}. And, as has been noted in \cite[p. 7]{Geng}, a geometry inducing maximal geometric structures on a manifold $M$ need not be unique. We now recall the following theorem of Wolf \cite{Wolf}.  
 
\begin{thm}[cf. \cite{Wolf} Theorem 2.4.17]\label{thm:HomogeneousCriterion}
Let $(N, h)$ be a Riemannian manifold and $p: (\widetilde{N}, \widetilde{h}) \to (N, h)$ be its universal cover, with deck transformation group $\Gamma \simeq \pi_1(N)$, and let $Z_{\operatorname{Isom}(\widetilde{N}, \widetilde{h})^o}(\Gamma)^o$ denote the connected component of the centralizer of $\Gamma$ in the connected component of the isometry group of $(\widetilde{N}, \widetilde{h})$. Then, $N$ is homogeneous if and only if $Z_{\operatorname{Isom}(\widetilde{N}, \widetilde{h})^o}(\Gamma)^o$ acts transitively on $\widetilde{N}$. 
\end{thm}
\noindent With this in hand, the following proposition is immediate.
\begin{prop}\label{prop:HomogeneousCriterion}
Let $(M,g)$ be a compact locally homogeneous manifold with universal Riemannian covering $(\widetilde{M}, \widetilde{g})$ and let $(\widetilde{M}, G)$ be a geometry on the universal cover $\widetilde{M}$ such that $\Isom(\widetilde{M}, \widetilde{g})^o \leq G$. If $(M,g)$ is homogeneous, then every $(\widetilde{M}, G)$-geometric structure on $M$ is homogeneous. 
\end{prop}
\noindent For our purposes, we will be interested in the following instance of the preceding proposition.
\begin{cor}
Let $G$ be a compact simple Lie group. A left-invariant metric $g$ on $G$ covers a homogeneous metric on $\Gamma \backslash G$, where $\Gamma \leq \Isom(G,g)$, only if the bi-invariant metric on $G$ induces a homogeneous metric on $\Gamma \backslash G$.
\end{cor}

\begin{proof}
By Theorem~\ref{thm:IsometryGroupofSimpleLieGroup2} the connected component of the identity of an arbitrary left-invariant metric on $G$ is contained in the connected component of the identity of a bi-invariant metric on $G$. We can then apply the preceding proposition.
\end{proof}

We now specialize to the compact simple Lie group $S^3$. Given a left-invariant metric $g$ on $S^3$ and a finite subgroup $\Gamma \leq \operatorname{O}(4)$ of Type I-VI acting freely on $S^3$,  let $g_\Gamma$ denote the corresponding locally homogeneous quotient metric on $\Gamma \backslash S^3$. In the case where $\Gamma = \Gamma_{q; p_1, p_2}$, where $p_1$ and $p_2$ are integers both relatively prime to the positive integer $q$, we will also denote the induced quotient metric on $L(q;p_1,p_2) = \Gamma_{q;p_1,p_2} \backslash S^3$ by $g_{q;p_1,p_2}$. Now notice that the spherical space form $\Gamma \backslash \mathbb{S}^3$ is homogeneous if and only if (up to conjugation in $O(4)$) $\Gamma$ is trivial, $\mathbb{Z}_2$, $\Gamma_{q;1, \pm 1}$ ($q \geq 3$), binary dihedral or binary polyhedral \cite[Corollary 2.7.2]{Wolf}. This leads to the following classification of homogeneous elliptic three-manifolds, which will be used in the proof of Theorem~\ref{thm:LocHomogEllipticThreeManifolds}.

\begin{prop}\label{lem:Centralizers}
Let $H \leq \SO(4)$ be the connected component of an isometry group of a
left-invariant metric $g$ on $S^3$ (see Proposition~\ref{prop:IsometryGroupsS^3}),  ${\bf P}^* \leq S^3$ be a binary polyhedral or binary dihedral group, and $\widetilde{\Delta}\leq S^3$ a group isomorphic to ${\bf P}^* \times 1$, $\widetilde{{\bf D}}_{2^{k+1}n}$ or  $\widetilde{{\bf T}}_{3^{k}8}$. 
Finally, let $\Gamma \leq H$ be of Type I-VI and $g_\Gamma$ the corresponding locally homogeneous metric on $\Gamma \backslash S^3$.

\begin{enumerate}

\item If $H = \rho(S^3 \times S^3) = \SO(4)$; i.e., $H$ is the identity component of the isometry group of a bi-invariant metric, then 
$$Z_{H}(\Gamma)^o = \left\{\begin{array}{ll}
\rho(S^3 \times S^3) \mbox{ and } g_\Gamma \mbox{ is homogeneous,} & \mbox{ if } \Gamma = 1, or\, \Gamma=\Gamma_{2;1,1} = \Gamma_{2;1,-1}, \\
\rho(S^3 \times S^1_{\rm{std}}) \mbox{ and } g_\Gamma \mbox{ is homogeneous,} & \mbox{ if } \Gamma = \Gamma_{q; 1, -1}, \; q \geq 3, \\
\rho(S^1_{\rm{std}} \times S^3) \mbox{ and } g_\Gamma \mbox{ is homogeneous,} & \mbox{ if } \Gamma = \Gamma_{q;1,1}, \; q \geq 3, \\
\rho(S^1_{\rm{std}} \times S^1_{\rm{std}}) \mbox{ and } g_\Gamma \mbox{ is not homogeneous,} & \mbox{ if } \Gamma = \Gamma_{q;1,p}, \; q \geq 3, \; p \not\equiv \pm 1 \mod q,\\
\rho(S^3 \times 1) \mbox{ and } g_\Gamma \mbox{ is homogeneous,} & \mbox{ if } \Gamma = \rho(1 \times {\bf P}^*),\\
\rho(1 \times S^3) \mbox{ and } g_\Gamma \mbox{ is homogeneous,} & \mbox{ if } \Gamma = \rho({\bf P}^*\times 1), \\
\rho(1 \times S^1_{\rm{std}}) \mbox{ and } g_\Gamma \mbox{ is not homogeneous,} & \mbox{ if } \Gamma = \rho(\widetilde{\Delta} \times {\bf C}_q) \mbox{ for } (q, |\widetilde{\Delta}|) =1, \\
\rho(S^1_{\rm{std}} \times 1) \mbox{ and } g_\Gamma \mbox{ is not homogeneous,} & \mbox{ if } \Gamma = \rho({\bf C}_q \times \widetilde{\Delta}) \mbox{ for } (q, |\widetilde{\Delta}|) =1.
\end{array}
\right.$$

\item If $H = \rho(S^3 \times S^1)$; i.e., each conjugate of $H$ in $\SO(4)$ is the identity component of the isometry group of a Berger metric that is not bi-invariant, then 
$$Z_{H}(\Gamma)^o = \left\{\begin{array}{ll}
\rho(S^3 \times S^1_{\rm{std}}) \mbox{ and } g_\Gamma \mbox{ is homogeneous,} & \mbox{ if } \Gamma = 1, or\, \Gamma=\Gamma_{2;1,1} = \Gamma_{2;1,-1};\\
\rho(S^3 \times S^1_{\rm{std}}) \mbox{ and } g_\Gamma \mbox{ is homogeneous,} & \mbox{ if } \Gamma = \Gamma_{q; 1, -1}, \; q \geq 3;\\
\rho(S^{1}_{\rm{std}} \times S^{1}_{\rm{std}}) \mbox{ and } g_\Gamma \mbox{ is not homogeneous,} & \mbox{ if } \Gamma = \Gamma_{q;1,1}, \, q \geq 3;\\
\rho(S^1_{\rm{std}} \times S^1_{\rm{std}}) \mbox{ and } g_\Gamma \mbox{ is not homogeneous,} & \mbox{ if } \Gamma = \Gamma_{q;1,p}, \; q \geq 3, \; p \not\equiv \pm1 \mod q;\\
\rho(1 \times S^1_{\rm{std}}) \mbox{ and } g_\Gamma \mbox{ is not homogeneous,} & \mbox{ if } \Gamma = \rho(\widetilde{\Delta} \times {\bf C}_q) \mbox{ for } (q, |\widetilde{\Delta}|) =1.
\end{array}
\right.$$

\item If $H = \rho(S^3 \times 1)$; i.e., each conjugate of $H$ in $\SO(4)$ is the identity component of the isometry group of a left-invariant metric that is not a Berger metric, then 
$$Z_{H}(\Gamma)^o = \left\{\begin{array}{ll}
\rho(S^3 \times 1) \mbox{ and } g_\Gamma \mbox{ is homogeneous,} & \mbox{ if } \Gamma = 1, or\, \Gamma=\Gamma_{2;1,1} = \Gamma_{2; 1,-1};\\
\rho(S^1_{\rm{std}} \times 1) \mbox{ and } g_\Gamma \mbox{ is not homogeneous,} & \mbox{ if } \Gamma = \Gamma_{q;1,1}, \; q \geq 3;\\
\rho(1 \times 1) \mbox{ and } g_\Gamma \mbox{ is not homogeneous,} & \mbox{ if } \Gamma = \rho({\bf P}^*\times 1). 
\end{array}
\right.$$

\end{enumerate}

\end{prop}

\begin{proof}
The proposition follows directly from the preceding discussion, the fact that a subgroup $K \leq S^3 \times S^3$ acts transitively on $S^3$ if and only if $\rho(K)\leq \SO(4)$ acts transitively on $S^3$, and the following lemma.
\begin{lem}
Let $\rho(H)$ be the isometry group of $(S^3, g)$, where $g$ is a left-invariant metric. Suppose that $\Gamma$ is a discrete subgroup of $\rho(H)$, then $$Z_{\rho(H)}(\Gamma)^o=\rho(Z_{H}(\widetilde{\Gamma})^o),$$ where $\widetilde{\Gamma}=\rho^{-1}(\Gamma)$.
\end{lem}
\begin{proof}
Take $h \in Z_{H}(\widetilde{\Gamma})^o$. By the definition of $\widetilde{\Gamma}$, $\rho(h)$ clearly centralizes $\Gamma$. Also, given any continuous curve $h_t$ connecting $h$ and the identity, $\rho(h_t)$ is a curve connecting $\rho(h)$ and the identity. Hence $Z_{\rho(H)}(\Gamma)^o\supset \rho(Z_{H}(\widetilde{\Gamma})^o)$.

To prove the other direction, let $\rho (h) \in Z_{\rho(H)}(\Gamma)^o$. Then there exists a curve $\rho(h)_t$ in $ Z_{\rho(H)}(\Gamma)$ connecting the identity element $e$ and $\rho(h)$. The curve $\rho(h)_t$ lifts to a curve $h_t$ starting from either $e$ or $-e$, and ends at $h$ or $-h$. By the continuity, $h_t$ must centralize $\Gamma$. Hence $h_t\in \pm Z_{H}(\widetilde{\Gamma})^o$ and $Z_{\rho(H)}(\Gamma)^o\subset \rho(\pm Z_{H}(\widetilde{\Gamma})^o)= \rho( Z_{H}(\widetilde{\Gamma})^o).$ This proves the lemma.
\end{proof}
\end{proof}

We now turn to the classification of locally homogeneous elliptic three-manifolds.

\begin{thm}\label{thm:LocHomogEllipticThreeManifoldsV2}
Let $g$ be a left-invariant metric on $S^3$ and denote its isometry class by $[g]$. Let $\Gamma \leq \Isom(S^3, g)^o \leq \SO(4)$ be a subgroup of Type I-VI acting freely on $S^3$ and let $\widetilde{\Delta}$ be equal to ${\bf P}^* \times 1$, $\widetilde{{\bf D}}_{2^{k+1}n}$ or $\widetilde{{\bf T}}_{3^{k}8}$, where ${\bf P}^* \leq S^3$ is binary dihedral or binary polyhedral. And, we take $p$ and $q$ to be positive integers.

\begin{enumerate}
\item Suppose $g$ is a bi-invariant metric and, hence, a metric of constant sectional curvature. Then, $[g]$ induces a unique isometry class $[g_\Gamma]$ of a locally homogeneous Riemannian metrics on $\Gamma \backslash S^3$. Up to scaling, $[g_\Gamma]$ is the unique isometry class of metrics of constant positive sectional curvature on $\Gamma \backslash S^3$ and  $[g_\Gamma]$ is homogeneous if and only if, up to conjugation in $\operatorname{O}(4)$, $\Gamma = \Gamma_{q;1,p}$, where $p \equiv \pm 1 \mod q$, or $\Gamma$ is binary dihedral or binary polyhedral.

\item Suppose $g$ is a Berger metric of non-constant sectional curvature. Then,
\begin{enumerate}
\item if, up to conjugation in $\operatorname{O}(4)$, $\Gamma$ is trivial, $\Gamma_{2;1,1}= \Gamma_{2;1,-1}$ or ${\bf P}^{*} \times 1$, then $[g]$ induces a unique isometry class $[g_\Gamma]$ of locally homogeneous Riemannian metrics on $\Gamma \backslash S^3$ and this class is homogeneous if and only if $\Gamma$ is trivial or $\Gamma_{2;1,1} = \Gamma_{2;1,-1}$ (i.e., $\Gamma \backslash S^3$ is $S^3$ or $\SO(3)$);

\item if, up to conjugation in $\operatorname{O}(4)$, $\Gamma = \rho(\widetilde{\Delta} \times {\bf C}_q)$, where $q \geq 2$, $(q, |\widetilde{\Delta}|)=1$  and $\widetilde{\Delta} \times \textcolor{red}{{\bf C}_q}$ is the subgroup of $S^3 \times S^{1}_{\rm{std}}$ generated by $\widetilde{\Delta}$ and $1 \times {\bf C}_q$, or $\Gamma = \rho(\widetilde{\Delta})$, where $\widetilde{\Delta} \neq {\bf P}^* \times 1$, then $[g]$ induces a unique isometry class $[g_\Gamma]$ of locally homogeneous Riemannian metrics on $\Gamma \backslash S^3$ and this class is not homogeneous;

\item if, up to conjugation in $\operatorname{O}(4)$, $\Gamma = \Gamma_{q;1,1}$ or $\Gamma_{q;1,-1}$ for $q \geq 3$, then $[g]$ induces precisely two isometry classes $[g_{\Gamma}^{+}]$ and $[g_{\Gamma}^{-}]$ of locally homogeneous Riemannian metrics on $\Gamma \backslash S^3 \simeq L(q;1,1) \simeq L(q;1,-1)$, precisely one of which is homogeneous;

\item if, up to conjugation in $\operatorname{O}(4)$, $\Gamma = \Gamma_{q;1,p}$ or $\Gamma_{q;1,-p}$, for $q \geq 3$ and $p \not\equiv \pm 1 \mod q$, then $[g]$ induces precisely two isometry classes $[g_{\Gamma}^{+}]$ and $[g_{\Gamma}^{-}]$ of locally homogeneous Riemannian metrics on $\Gamma \backslash S^3 \simeq L(q;1,p) \simeq L(q;1,-p)$, neither of which is homogeneous.

\end{enumerate}

\item Suppose $g$ is not a Berger metric (i.e., $g$ is a generic left-invariant metric). 
\begin{enumerate}

\item if, up to conjugation in $\operatorname{O}(4)$, $\Gamma = \Gamma_{q;1,1}$, then $[g]$ induces a unique isometry class $[g_\Gamma]$ of locally homogeneous Riemannian metrics on $\Gamma \backslash S^3 = L(q;1,1)$ and this class is homogeneous if and only if $q$ equals $1$ or $2$ (i.e., $\Gamma \backslash S^3$ is $S^3$ or $\SO(3)$).

\item if, up to conjugation  in $\operatorname{O}(4)$, $\Gamma$ is binary dihedral or binary polyhedral, then $[g]$ induces a unique isometry class $[g_\Gamma]$ of locally homogeneous Riemannian metrics on $\Gamma \backslash S^3$ and this class is not homogeneous.
 \end{enumerate}
\end{enumerate}
\end{thm}

\begin{proof} 
$(1)$ This statement is due to Wolf \cite[Theorem 7.6.6]{Wolf}. 

$(2)$ Assume that $g$ is a Berger metric of non-constant sectional curvature. We first prove 2(a). Begin by fixing $\Gamma$ to be trivial, ${\bf P}^* \times 1$, or $\Gamma_{2;1,1} = \Gamma_{2;1,-1} = \langle -I_4 \rangle$. Then, for any left-invariant metric $h \in [g]$ we see that $\Gamma \leq \Isom(S^3, h)^o = \rho(S^3 \times S^{1}(h))$, where $S^{1}(h) \leq S^3$ is the appropriate one-parameter subgroup (see Computation~\ref{comp:NCCBergerMetrics}). Therefore, each left-invariant metric $h \in [g]$ induces a locally homogeneous metric $h_\Gamma$ on $\Gamma \backslash S^3$ and by Lemma~\ref{lem:Centralizers}(2) this metric cannot be homogeneous. We claim that different left-invariant metrics in $[g]$ give isometric metrics on $\Gamma \backslash S^3$. To prove the claim, let $h, h' \in [g]$, Then there exists $x \in S^3$ such that $C_x: (S^3, h) \to (S^3, h')$ is an isometry. Now let $F = L_{x^{-1}} \circ C_{x}: (S^3, h) \to (S^3, h')$. Then $F$ is an isometry that normalizes $\Gamma$: $\Gamma^{F} = \Gamma$. By Lemma~\ref{lem:Basic}, $F$ induces an isometry between $(\Gamma \backslash S^3, h_\Gamma)$ and $(\Gamma \backslash S^3, h_{\Gamma}')$, which proves the claim.

Now, we show that for any $\Gamma'$ conjugate to $\Gamma$, the Riemannian manifold $(\Gamma' \backslash S^3, h_{\Gamma'})$ is isometric to $(\Gamma \backslash S^3, h_{\Gamma})$ for all $h\in [g]$. Suppose that $\Gamma'$ is conjugate to $\Gamma$ in $\operatorname{O}(4)$ and that $\Gamma'$ is a subgroup of $\rho(S^3\times S^1(h))$ for some $h\in [g]$, then by Theorem \ref{SeifertThrelfalltheorem}, $\Gamma=\Gamma'^{\rho(a, b)}$ for some $(a, b)\in S^3\times S^3$. Therefore, 
$\Gamma$ is a subgroup of $\rho(S^3\times S^1(h)^{(a, b)})=\rho(S^3\times S^1(h_1))$ for some $h_1 \in [g]$. As $h$ and $h_1$ are in $[g]$, the previous paragraph and Lemma \ref{lem:Basic} imply that $(\Gamma' \backslash S^3, h_{\Gamma'})$ is isometric to $(\Gamma \backslash S^3, h_{\Gamma})$, which completes the proof of  2(a).

We now move to the proof of 2(b). Suppose $\Gamma \leq \SO(4)$ is of Type II-VI but not binary dihedral or binary polyhedral. Then, it follows from Proposition~\ref{prop:IsometryGroupsS^3} that there is a unique $h \in [g]$ such that $\Gamma \leq \Isom(S^3, h)^0 = \rho(S^3 \times S^{1}(h))$. The metric $h$ induces a locally homogeneous metric $h_\Gamma$  on $\Gamma \backslash S^3$ which cannot be homogeneous by Lemma~\ref{lem:Centralizers}(2). Again, by Lemma \ref{lem:Basic}, if $\Gamma \leq \rho(S^3\times S^1(h))$ and $\Gamma'\leq \rho(S^3\times S^1(h'))$ are conjugate in $\operatorname{O}(4)$, $(\Gamma \backslash S^3, h_\Gamma)$ and $(\Gamma' \backslash S^3, h_{\Gamma'}')$ must be isometric. This establishes statement 2(b).

By Lemma~\ref{lem:Basic}, to prove 2(c), it suffices to consider $\Gamma = \Gamma_{q;1,p}$ where $q \geq 3$. Now, let $h \in [g]$ be such that $\Isom(S^3, h)^o = \rho(S^3 \times S^{1}_{\rm{std}})$, then $h$ is the unique left-invariant metric in $[g]$ such that $\Gamma_{q;1,p}, \Gamma_{q;1,-p} \leq \Isom(S^3,h)^o$. Then, $h$ induces the locally homogeneous metrics $h_{q;1,p}$ and $h_{q;1,-p}$ on $\Gamma \backslash S^3 \simeq L(q;1,p) \simeq L(q;1,-p)$ (see p.~\pageref{text:LensSpace} for notation). Now, recall that $\Gamma_{q;1,p} = \Gamma_{q;1, p^{-1}}^{F_1}$, where
$$F_1 =  \left(\begin{array}{cc}0 & I_2 \\I_2 & 0\end{array}\right) \in \SO(4),$$
(see Remark~\ref{rem:TypeIConjugacy}). Noticing that $F_1 \in N_{\operatorname{Diff}(S^3)}(\Isom(S^3, h)^o)$ we conclude from Lemma~\ref{lem:Basic} that 
$$\overline{F}_1^{*}( \mathscr{R}_{\rm{L-loc}}^{\Gamma_{q;1,p^{-1}}}(L(q;1, p^{-1})) ) = \mathscr{R}_{\rm{L-loc}}^{\Gamma_{q;1, p}}(L(q;1, p))$$
and 
$$\overline{F}_1^{*}( \mathscr{R}_{\rm{L-loc}}^{\Gamma_{q;1,-p^{-1}}}(L(q;1, -p^{-1})) ) = \mathscr{R}_{\rm{L-loc}}^{\Gamma_{q;1, -p}}(L(q;1, -p)).$$ 
Now, imagine there is an isometry $\Psi : (L(q;1,p), h_{q;1,p}) \to (L(q;1,-p), h_{q;1,-p})$, then we obtain an isometry $\widetilde{\Psi} \in  \Isom(S^3, h)$ such that $\pi_{q;1,-p} \circ \widetilde{\Psi} = \Psi \circ \pi_{q;1,p}$; where, for any positive integer $c$, $\pi_{q;1,p}: S^3 \to L(q;1,c) = \Gamma_{q;1,c} \backslash S^3$ is the natural projection map.
Then, $\widetilde{\Psi} \in \Isom(S^3, h) \leq \operatorname{O}(4)$ and $\Gamma_{q;1, p}^{\widetilde{\Psi}} = \Gamma_{q;1,-p}$. Therefore, $\widetilde{\Psi} \in F_2 N_{\operatorname{O}(4)}(\Gamma_{q;1,p})$, where 
$$F_2 = I_2 \oplus \left(\begin{array}{cc}0 & -1 \\ -1 & 0\end{array}\right) \in \operatorname{O}(4) - \SO(4),$$
(see Remark~\ref{rem:TypeIConjugacy}). Noticing that $N_{\operatorname{O}(4)}(\Gamma_{q;1,p}) = N_{\SO(4)}(\Gamma_{q;1,p})$, we conclude that $\widetilde{\Psi} = F_2 A$, where $A \in \SO(4)$. Therefore, conjugating $\Isom(S^3, h)^o = \rho(S^3 \times S^{1}_{\rm{std}})$ by $\widetilde{\Psi}$ we obtain $\rho(S^{1}_{\rm{std}} \times S^3)$ (see Remark~\ref{rem:LeftRight}) contradicting the fact that $\widetilde{\Psi}$ is an isometry of $(S^3, h)$. Consequently, $(L(q;1,p), h_{q;1,p})$ and $(L(q;1,-p), h_{q;1,-p})$ are locally isometric, yet non-isometric. 

Now, in the event that $p \equiv \pm 1 \mod q$, Lemma~\ref{lem:Centralizers}(2) implies that precisely one of the metrics $h_{q;1,p}$ and $h_{q;1,-p}$ is homogeneous. Thus, establishing statement (2c). And, in the case where $p \not\equiv \pm 1 \mod q$, Lemma~\ref{lem:Centralizers}(2) shows that neither of the metrics $h_{q;1,p}$ and $h_{q;1,-p}$ is homogeneous. Thus, statement 2(d) is established.
 
We conclude with the proof of statement 3. Let $g$ be a left-invariant metric that is not naturally reductive. Then, for each left-invariant $h \in [g]$ we have $\Gamma_{q;1,1} \leq \rho(S^3 \times 1) = \Isom(S^3, h)^o$. Therefore, any $h \in [g]$ induces a locally homogeneous metric on $\Gamma \backslash S^3$ that cannot be homogeneous by Lemma~\ref{lem:Centralizers}(3). Now, for $h , h' \in [g]$ there is an $x \in S^3$ such that conjugation by $x$ is an isometry: $C_x : (S^3, h) \to (S^3, h')$. Now, let $F = L_{x^{-1}} \circ C_{x}: (S^3, h) \to (S^{3}, h')$. Then $F$ is an isometry such that $\Gamma^{F} = \Gamma$. By Lemma~\ref{lem:Basic}, $F$ induces an isometry between $(\Gamma \backslash S^3, h_\Gamma)$ and $(\Gamma \backslash S^3, h_{\Gamma}')$. This establishes statement 3(a). The proof of 3(b) is analogous to the proof of 2(a).
\end{proof}

\begin{proof}[Proof of Theorem~\ref{thm:LocHomogEllipticThreeManifolds}]
This is merely a reformulation of Theorem~\ref{thm:LocHomogEllipticThreeManifoldsV2}.
\end{proof}

%%%%%%%%%%%%%%%%%
%%%%%%%%%%%%%%%%%
%%%%%%%%%%%%%%%%%
\section{Classification of the compact homogeneous three-manifolds} \label{sec:CptHomogeneous}
In this section we will prove Theorem~\ref{thm:HomogeneousThreeManifolds} which classifies the compact homogeneous Riemannian three-manifolds. 

\begin{lem}\label{lem:HomogeneousQuotients}
Let $(\widetilde{M}, \widetilde{g})$ be a simply-connected homogeneous space and $\operatorname{Isom}(\widetilde{M}, \widetilde{g})^o$ the identity component of its isometry group.
$(\widetilde{M}, \widetilde{g})$ admits a compact quotient that is homogeneous if and only if there is a discrete abelian subgroup $\Gamma \leq \operatorname{Isom}(\widetilde{M}, \widetilde{g})^o$ acting freely and properly discontinuously on $\widetilde{M}$ such that the corresponding Riemannian quotient $(\widetilde{M}/\Gamma, \widetilde{g}_{\Gamma})$ is compact and homogeneous. 
\end{lem}

\begin{proof}
Let $(M,g)$ be a compact homogeneous quotient of $(\widetilde{M},\widetilde{g})$.  As $M$ is homogenous, there is a connected Lie group $G$ and a closed subgroup $K$ such that $M=G/K$. As connected Lie groups have abelian fundamental groups, the homotopy long exact sequence implies that the fundamental group of $M$ has a finite index abelian subgroup isomorphic with a subgroup $\Lambda \leq \operatorname{Isom}(\widetilde{M},\widetilde{g})$ acting properly discontinuously and freely by deck-transformations on  $\widetilde{M}$. Let $\Gamma=\Lambda \cap \operatorname{Isom}(\widetilde{M},\widetilde{g})^o$.  The group $\Gamma$ has finite index in $\Lambda$ since the isometry group of a homogenous space has finitely many components.  The compact quotient $\Gamma \backslash \widetilde{M}$ is homogenous by Theorem~\ref{thm:HomogeneousCriterion}.
\end{proof}

\begin{proof}[Proof of Theorem~\ref{thm:HomogeneousThreeManifolds}] 
Let $(M,g)$ be a compact homogenous three-manifold.  Its universal Riemannian cover $(\widetilde{M},\widetilde{g})$ is homogenous by Singer \cite[p. 692]{Singer}.  By Lemma \ref{lem:HomogeneousQuotients}, the identity component of the isometry group $\operatorname{Isom}(\widetilde{M},\widetilde{g})^0$ contains a discrete abelian subgroup $\Gamma$ that acts on $\widetilde{M}$ with compact homogenous quotient.  Moreover, $(\widetilde{M}, \operatorname{Isom}(\widetilde{M},\widetilde{g}))$ is modeled on a subgeometry of one of the eight Thurston geometries (see \cite{Scott} and \cite{LSS}).    The Thurston geometries admitting compact quotients with abelian fundamental groups are precisely the geometries $(S^3, \Isom(\mathbb{S}^3)^o)$, $(\R^3, \Isom(\mathbb{E}^3)^o)$, and $(S^2 \times \R, \Isom(\mathbb{S}^2 \times \mathbb{E})^o)$  (cf. \cite[Table 1 and 2]{AFW}). 

The case where $(M,g)$ is a homogeneous three-manifold for which the locally homogeneous metrics on $M$ are modeled on subgeometries of $(S^3, \Isom(\mathbb{S}^3)^o)$ is handled by Theorem~\ref{thm:LocHomogEllipticThreeManifolds}. 

Next, assume $(M,g)$ is a homogeneous three-manifold modeled on a subgeometry of the three-dimensional euclidean geometry $(\R^3, \Isom(\mathbb{E}^3)^o)$. Then, the isometry group of the universal Riemannian cover of $(M,g)$ satisfies $\Isom(\R^3, \widetilde{g})^o \leq \Isom(\mathbb{E}^3)^o$, and we conclude that $M = \Gamma \backslash \R^3$ for some discrete subgroup $\Gamma \leq \Isom(\R^3, \widetilde{g})$ acting freely. Then, by Lemma~\ref{lem:HomogeneousQuotients}, $\Gamma$ possesses a finite index subgroup $\Lambda \leq \Isom(\R^3, \widetilde{g})^o$ such that the quotient space $(\Lambda \backslash \R^3, \widetilde{g}_{\Lambda})$ is homogeneous. Now, since $\Isom(\R^3, \widetilde{g})^o$ is contained in $\Isom(\mathbb{E}^3)^o$, Theorem~\ref{thm:HomogeneousCriterion} confirms that $\Lambda \backslash \mathbb{E}^3$ is also homogeneous. Hence, by the Bieberbach theorems, $\Lambda$ is a lattice of full rank in $\R^3$ (which is identified with the group of translations of euclidean three-space). Then, $Z_{\Isom(\mathbb{E}^3)}(\Lambda)^o = \R^3 \geq Z_{\Isom(\R^3, \widetilde{g})^o}(\Gamma)^o$; that is, $Z_{\Isom(\R^3, \widetilde{g})^o}(\Gamma)^o$ consists of translations of euclidean three-space. As $(M,g)$ is assumed to be homogeneous, Theorem~\ref{thm:HomogeneousCriterion} implies the group $Z_{\Isom(\R^3, \widetilde{g})^o}(\Gamma)^o$ must act transitively on $\R^3$. Therefore, since it consists of translations of three-space, $Z_{\Isom(\R^3, \widetilde{g})^o}(\Gamma)^o$ is identically $\R^3$. Consequently, $\widetilde{g}$ is a translation-invariant metric on $\R^3$ and, therefore, flat. Using \cite[Proposition III.6.6]{Sakai2} we conclude that the homogeneous space $(M,g)$ is a flat torus.

Finally, the compact quotients of $S^2 \times \R$ are $M_1 = S^2 \times S^1$, $M_2$ (the non-trivial $S^1$-bundle over $\mathbb{R}P^2$), $M_3 = \R P^2 \# \R P^2$, and $M_4 = \R P^2 \times S^1$. These spaces are precisely the compact three-manifolds whose geometric structures are modeled on subgeometries of $(S^2 \times \R, \Isom(\mathbb{S}^2 \times \mathbb{E})^o)$. In fact, all geometric structures on the manifolds $M_1$, $M_2$, $M_3$ and $M_4$ are maximal. Using the notation of \cite[Sec. 4]{LSS}, the locally homogeneous metrics supported by these spaces give us the following families of locally homogeneous metrics: $M_1(k,v), M_2(k,v), M_3(k,v) \mbox{ and } M_4(k,v)$, where $k, v >0$, respectively. Carrying out the computation suggested by Lemma~\ref{lem:HomogeneousQuotients}, we conclude $M_1(k,v)$, $M_2(k,v)$ and $M_4(k,v)$ are homogeneous, while $M_3(k,v)$ is merely locally homogeneous, for any $k,v >0$. We note that $M_1(k,v)$ and $M_4(k,v)$ are globally symmetric, while $M_2(k,v)$ and $M_3(k,v)$ are only locally symmetric.
 \end{proof}

\begin{rem}
As the elliptic three-manifold $\Gamma_{q; 1,1} \backslash S^3 \simeq \Gamma_{q; 1,-1} \backslash S^3$, where $q \geq 3$, admits locally isometric locally homogeneous metrics where one is homogeneous and the other is not, we see that the methods employed in the following section cannot resolve whether homogeneity is audible among compact locally homogeneous three-manifolds. 
\end{rem}

%%%%%%%%%%%%%%%%%%%%%%%%%%%%%%%%%%%%%%%%%
%%%%%%%%%%%%%             Heat Invariants on S^3     %%%%%%%%%%%
%%%%%%%%%%%%%%%%%%%%%%%%%%%%%%%%%%%%%%%%%
\section{On hearing the geometry of a locally homogeneous elliptic three-manifold}\label{sec:SphericalHeatInvariants}

Our objective in this section is to establish Theorem~\ref{thm:AudibillityLocalGeometry} which provides strong evidence that among locally homogeneous three-manifolds a locally homogeneous elliptic three-manifold is determined up to local isometry by its first four heat invariants. 

The first step of our argument is to observe that the first four heat invariants of a locally homogeneous three-manifold Riemannian covered by a unimodular Lie equipped with a left-invariant metric can be expressed as symmetric functions in the Christoffel symbols $\mu_1 = \Gamma_{12}^3 = -\Gamma_{21}^3$, $\mu_2 = \Gamma_{23}^{1} = -\Gamma_{32}^1$ and $\mu_3 = \Gamma_{31}^2 = -\Gamma_{13}^2$ (see Theorem~\ref{thm:HeatInvariantsSymmetricPolysChristoffel}). We then explore the degree to which this process can be inverted and learn from Lemmas~\ref{lem:SymmetricPolynomialsHeatInvariants} and \ref{lem:P1HeatInvariants} that for a locally homogeneous elliptic three-manifold $(\Gamma \backslash S^3, g)$ the elementary symmetric polynomials in $\mu_1$, $\mu_2$ and $\mu_3$ can be expressed as functions of the first four heat invariants and the order of $\Gamma$ which, by Lemma~\ref{lem:SymmetricPolynomialsUsefulFact}, establishes that the Christoffel symbols can be recovered from this data. Coupled with the fact that Equation~\ref{eq:ChristoffelToMetricEigenvalues} shows the Christoffel symbols of $(\Gamma \backslash S^3, g)$ determine its metric eigenvalues (see Section~\ref{sec:IsometryClassesOfLeftInvariantMetrics}), this shows that two locally homogeneous elliptic three-manifolds $(\Gamma_1 \backslash S^3, g_1)$ and $(\Gamma_2 \backslash S^3, g_2)$ having fundamental groups of the same order and sharing the same first four heat invariants must be locally isometric (see Theorem~\ref{thm:FundamentalGroupSameSize}). Then, by using the classification theorem of locally homogeneous three-manifolds (Theorem~\ref{thm:LocHomogEllipticThreeManifolds}), we are able to deduce Theorem~\ref{thm:IsospectralSets} and \ref{thm:A0A1A2SufficeScalarFlat}. Hence, isospectral locally non-isometric locally homogeneous elliptic three-manifolds have fundamental groups of different orders. Therefore, since the first four heat invariants can decipher whether the Ricci tensor of a locally homogeneous elliptic three-manifold is degenerate \cite[Theorem 1.1(2)]{LSS}, Lemma~\ref{lem:DegenerateRicci} implies the respective Ricci tensors of $(\Gamma_1 \backslash S^3, g_1)$ and $(\Gamma_2 \backslash S^3, g_2)$ are non-degenerate. With this in hand, the proof of Theorem~\ref{thm:AudibillityLocalGeometry} follows easily.

For the remainder of this article we will let
$$P_1(x,y,z) \equiv x+y+z, P_2(x,y,z) \equiv xy +xz +yz \mbox{ and } P_3(x,y,z) \equiv xyz$$
be the elementary symmetric polynomials in three variables. And, we recall the following useful fact.

\begin{lem}\label{lem:SymmetricPolynomialsUsefulFact}
The three-multiset $[\alpha, \beta, \gamma]$ determines and is determined by the ordered triple 
$(P_1(\alpha, \beta, \gamma), P_2(\alpha, \beta, \gamma), P_3(\alpha, \beta, \gamma)$.
\end{lem}

\begin{proof}
This follows by observing that 
$$(x+\alpha)(x+\beta)(x+\gamma) = x^3 +P_1(\alpha, \beta, \gamma) x^2 + P_2(\alpha, \beta, \gamma) x + P_3(\alpha, \beta, \gamma).$$
\end{proof}

\subsection{Heat invariants as symmetric polynomials in the Christoffel symbols}
As we noted in \cite{LSS}, the heat invariants of a compact Riemannian manifold $(M,g)$ are the coefficients of the asymptotic expansion of the heat trace \cite{Mina}: 
$$\operatorname{Tr}(e^{-t\Delta_g}) = \sum_{k=0}^{\infty} e^{-t \lambda_j}\stackrel{t \searrow 0}{\sim}(4 \pi t)^{-n/2} \sum_{m = 0}^{\infty} a_m(M,g) t^m.$$ 
Two isospectral manifolds necessarily have the same heat invariants, but the converse is false. The heat invariants are known to be defined by the formula (cf. \cite[p. 145]{Berard} or \cite[Chp. VI.5]{Sakai2}): 
$$a_m(M,g) = \int_{M} u_m(M,g) \, d\nu_g,$$ 
where $u_m(M,g)$ is a universal polynomial in the components of the curvature tensor and its covariant derivatives. Unfortunately, explicit formulae are known in only a few cases.

In \cite{LSS} we saw that for a compact locally homogeneous three-manifold $(M,g)$ modeled on a simply-connected unimodular Lie group equipped with a left-invariant metric, the heat invariants $a_1(M,g)$, $a_2(M,g)$ and $a_3(M,g)$ can be expressed as symmetric polynomials in the eigenvalues of the Ricci tensor. It will be useful for us to observe here that for such manifolds, the heat invariants can also be expressed as symmetric polynomials in the Christoffel symbols with respect to a Milnor frame.
  
 \begin{lem}\label{lem:HeatInvariantsSymmetricPolys}
Let $(M,g)$ be a locally homogeneous three-manifold locally isometric to a unimodular Lie group $G$ equipped with a left-invariant metric. Let $\{E_1, E_2, E_3\}$ be a Milnor frame (see \cite[Def. 2.13]{LSS}) and let $\mu = (\mu_1, \mu_2, \mu_3)$ be the associated vector of Christoffel symbols:  
$\mu_1 = \Gamma_{12}^3 = -\Gamma_{13}^2$, $\mu_2 = \Gamma_{23}^1 = -\Gamma_{21}^3$ and $\mu_3 = \Gamma_{31}^2 = - \Gamma_{32}^1$. 

\begin{enumerate}
\item The Christoffel symbols and principal curvatures are related as follows: 

\begin{equation}\label{eqn:PrincipalCurvaturesChristoffelSymbols}
 \left(\begin{array}{c}K_{12} \\K_{13} \\K_{23}\end{array}\right) = 
\left(\begin{array}{c}R(e_1, e_2, e_1, e_2) \\ R(e_1, e_3, e_1, e_3) \\ R(e_2, e_3, e_2, e_3)\end{array}\right)
 =  \left(\begin{array}{ccc}-1 & 1 & 1 \\1 & -1 & 1 \\1 & 1 & -1\end{array}\right) \left(\begin{array}{c} \mu_1 \mu_2 \\ \mu_1 \mu_3 \\ \mu_2 \mu_3 \end{array}\right).
 \end{equation} 
 
\item And, setting $P_j \equiv P_j(\mu)$ for $j=1,2,3$, we have the folowing expressions:  
 
 \begin{equation}\label{eq:S}
\Scal=2P_2,
\end{equation}  
\begin{equation}\label{eq:NormSquaredR}
\vert R \vert^2=12P_2^2 -32 P_1P_3,
\end{equation} 
\begin{equation}\label{eq:NormSquaredRho}
\vert \Ric \vert^2= 4P_2^2 - 8P_1P_3,
\end{equation} 
\begin{equation}\label{eq:(R,R,R)}
(R,R,R)=8(P_2^3 -24P_3^2),
\end{equation} 
\begin{equation}\label{eq:(rho;R,R)}
(\Ric;R,R) = -48P_3^2 + 8P_2^3 -16P_1P_2P_3,
\end{equation} 
\begin{equation}\label{eq:(rho;rho;R)}
(\Ric;\Ric;R) = 8(P_1P_2P_3 - 6P_3^2),
\end{equation} 
\begin{equation}\label{eq:(rhorhorho)}
(\Ric \Ric \Ric)=24P_3^2 + 8P_2^3 -24P_1P_2P_3
\end{equation} 
\begin{equation}\label{eq:NormSquaredNablaR}
\vert \nabla R \vert^2 = 32(-9P_3^2 -4P_1^3P_3 + P_1^2P_2^2 + 10P_1P_2P_3 -2P_2^3),
\end{equation} 
\begin{equation}\label{eq:NormSquaredNablaRho}
\vert \nabla \Ric \vert^2= 8 \left(-9P_3^2 -4P_1^3P_3 + P_1^2P_2^2 + 10P_1P_2P_3 -2P_2^3 \right).
\end{equation} 
\end{enumerate}

\end{lem}

\begin{proof}
The first statement follows from the fact that the Christoffel symbols $\mu = (\mu_1, \mu_2, \mu_3)$ are related to the eigenvalues of the Ricci tensor $\nu = (\nu_1, \nu_2, \nu_3)$ via $\nu_{\sigma(1)} = 2\mu_{\sigma(2)}\mu_{\sigma(3)}$ for any permutation of three elements, and in \cite[Cor. 2.36]{LSS} we saw that principal curvatures and Ricci eigenvalues are related as follows:
$$\left(\begin{array}{c}K_{12} \\K_{13} \\K_{23}\end{array}\right) 
 =  \frac{1}{2}\left(\begin{array}{ccc}-1 & 1 & 1 \\1 & -1 & 1 \\1 & 1 & -1\end{array}\right) \left(\begin{array}{c} \nu_3 \\ \nu_2 \\ \nu_1 \end{array}\right).$$
The remaining expressions are then a direct application of the relation between the principal curvatures and the Christoffel symbols coupled with \cite[Prop. 2.23]{LSS}.
\end{proof}

Keeping in mind that $\nabla \Scal$ is zero for a locally homogeneous space, one easily deduces the following expressions for the heat invariants, by applying Lemma~\ref{lem:HeatInvariantsSymmetricPolys} to \cite[Equations~(2.2) - (2.5)]{LSS}. 

\begin{thm}\label{thm:HeatInvariantsSymmetricPolysChristoffel}
Let $(M,g)$ be a compact locally homogeneous three-manifold modeled on a unimodular Lie group equipped with a left-invariant metric and let $\mu = (\mu_1, \mu_2, \mu_3)$ be the vector of Christoffel symbols associated to a Milnor frame on $(M,g)$. Then, the heat invariants may be computed in terms of the Christoffel symbols as follows

\begin{equation}\label{eq:A1Christoffel}
a_1(M,g) = \frac{a_0(M,g)}{3}P_2,
\end{equation}

\begin{equation}\label{eq:A2Christoffel}
a_2(M,g) = \frac{a_0(M,g)}{360}\left( 36P_2^2 - 48 P_1P_3 \right),
\end{equation}

and

\begin{equation}\label{eq:A3Christoffel}
a_3(M,g) = \frac{a_0(M,g)}{7!} \left( -240 P_3^2 - 576P_1P_2P_3 + 184 P_2^3 + 192P_1^3 P_3 - 48 P_1^2P_2^2\right),
\end{equation}
where $P_j \equiv P_j(\mu)$ for $j =1, 2,3$.
\end{thm}

\begin{proof}
Keeping in mind that $\nabla \Scal$ is zero for a locally homogeneous space, the result follows immediately from Lemma~\ref{lem:HeatInvariantsSymmetricPolys} applied to \cite[Eqs. (2.2) - (2.5)]{LSS}.
\end{proof}

%%%%%%%%%%%%%%%%%%%%%%%%%%%%%%
\subsection{Recovering the metric eigenvalues from the Christoffel symbols}
The Cartan-Killing form associated to $G = S^3 \simeq \SU(2)$ is the symmetric and negative-definite bilinear form $B(x,y)=\trace(\ad(x)\ad(y))$ for $x,y \in \germ{su}(2)$. Let $g_0$ denote the bi-invariant metric on $S^3$ induced by the inner product $ \Inner_0 \equiv -\frac{1}{2}B(\cdot, \cdot)$. Then, $(S^3, g_0)$ is the round sphere of radius $2$. It has volume $16\pi^2$ and constant sectional curvature $\frac{1}{4}$.

Now, fix an orientation on $\germ{su}(2)$ and let $\times$ denote the cross-product on the oriented inner product space $(\germ{su}(2), \Inner_0 )$, then for any $u,v, \in \germ{su}(2)$
$$[u,v] = u \times v.$$ 
Consider the oriented metric Lie algebra $(\germ{su}(2), \Inner)$  and let $g$ denote the corresponding left-invariant metric on $S^3$. Then, there is a self-adjoint map $\Omega : (\germ{su}(2), \Inner_0 )\rightarrow (\germ{su}(2), \Inner_0 )$ such that 
$\langle X, Y \rangle = \langle \Omega(X), Y \rangle_0$ for $X, Y \in \germ{su}(2)$.
Let $\{ e_1, e_2, e_3 \}$ be a positively oriented $\Inner_0$-orthonormal basis of $\Omega$-eigenvectors with corresponding nonzero eigenvalues $\eta_1^2$, $\eta_2^2$ and $\eta_3^2$. So, $\eta_1^2$, $\eta_2^2$ and $\eta_3^2$ are the ``metric eigenvalues'' of $\Inner$ with respect to $\Inner_0$. Then, $\{ \frac{e_1}{\eta_1}, \frac{e_2}{\eta_2}, \frac{e_3}{\eta_3} \}$ is a positively-oriented $\Inner$-orthonormal basis and 

\begin{equation}\label{eq:VolumeAndMetricEigenvalues}
\vol(g) = \eta_1 \eta_2 \eta_3 \vol(g_0) =  \eta_1 \eta_2 \eta_3 16\pi^2.
\end{equation}

\noindent
Let $\widetilde{\times}$ be the cross-product on the the oriented metric Lie algebra $(\germ{su}(2), \Inner)$ and let $L: (\germ{su}(2), \Inner) \to (\germ{su}(2), \Inner)$ be the associated ``Milnor map:'' $$[u,v] = L( u \widetilde{\times} v).$$
Then, for any cyclic permutation $ \sigma$, we find 

$$\frac{e_{\sigma(3)}}{\eta_{\sigma(1)} \eta_{\sigma(2)}} =  \frac{e_{\sigma(1)}}{\eta_{\sigma(1)}} \times \frac{e_{\sigma(2)}}{\eta_{\sigma(2)}} =  \left[\frac{e_{\sigma(1)}}{\eta_{\sigma(1)}} , \frac{e_{\sigma(2)}}{\eta_{\sigma(2)}} \right] = L \left( \frac{e_{\sigma(1)}}{\eta_{\sigma(1)}} \widetilde{\times} \frac{e_{\sigma(2)}}{\eta_{\sigma(2)}} \right) = L\left(\frac{e_{\sigma(3)}}{\eta_{\sigma(3)}} \right) = \frac{1}{\eta_{\sigma(3)}} L(e_{\sigma(3)}),
$$
and we deduce
$$L(e_1) = \frac{\eta_1^2}{\eta_1\eta_2\eta_3} e_1, \, L(e_2) = \frac{\eta_2^2}{\eta_1\eta_2\eta_3} e_2, \, \mbox{ and } L(e_3) = \frac{\eta_3^2}{\eta_1\eta_2\eta_3} e_3.$$
Therefore, $\Omega$ and $L$ have the same eigenvectors, and the eigenvalues of $\Omega$ determine and are determined by the eigenvalues of $L$.  

Letting $[e_i, e_j] = \sum_{j=1}^{3} \alpha_{ijk} e_k$ for constants $\alpha_{ijk}$, we may define 
$$c_{ijk} \equiv g([ \frac{e_i}{\eta_i} , \frac{e_j}{\eta_j}], \frac{e_k}{\eta_k}) = \alpha_{ijk}\frac{\eta_k^2}{\eta_i \eta_j \eta_k}.$$ Using Equation~\eqref{eq:VolumeAndMetricEigenvalues}, the non-zero $c_{ijk}$'s are:
$$c_{123} = - c_{213} = \eta_3^2 \frac{16\pi^2}{\vol(g)}, \;  c_{231} = - c_{321} = \eta_1^2 \frac{16\pi^2}{\vol(g)}, \;  
\mbox{ and } c_{321} = - c_{132} = \eta_2^2 \frac{16\pi^2}{\vol(g)}.$$
Applying Kozul's formula, the Christoffel symbols can be expressed as 
$$\Gamma_{ij}^{k} = \frac{1}{2}(c_{ijk} - c_{ikj} - c_{jki}),$$
which leads to the following relationship between the metric eigenvalues and the Christoffel symbols
\begin{equation}\label{eq:ChristoffelToMetricEigenvalues}
\left( \begin{array}{c} \eta_1^2 \\ \eta_2^2 \\ \eta_3^2  \end{array} \right) = \frac{\vol(g)}{16 \pi^2}\left( \begin{array}{ccc} 0 & 1 &1 \\ 1 & 0 & 1 \\ 1 & 1& 0  \end{array} \right) \left( \begin{array}{c} \mu_1 \\ \mu_2 \\ \mu_3 \end{array} \right),
\end{equation}
where as in Lemma~\ref{lem:HeatInvariantsSymmetricPolys}, $\mu_1 = \Gamma_{12}^3 = -\Gamma_{13}^2$, $\mu_2 = \Gamma_{23}^1 = -\Gamma_{21}^3$ and $\mu_3 = \Gamma_{31}^2 = -\Gamma_{32}^1$. Since the metric $g$ is determined up to (local) isometry by its 3-multiset $[\eta_1^2, \eta_2^2, \eta_3^2]$ of metric eigenvalues, we conclude $g$ is also determined up to (local) isometry by the $3$-multiset of its Christoffel symbols $[\mu_1, \mu_2, \mu_3]$. Keeping in mind that if the left-invariant metric $g$ on $S^3$ induces a locally homogeneous metric $g_\Gamma$ on the elliptic three-manifold $\Gamma \backslash S^3$, then $\vol(g) = |\Gamma| \vol(g_\Gamma)$, we see the preceding discussion establishes the following proposition. 

\begin{prop}\label{prop:DeterminedByChristoffelSymbols}
Among locally homogeneous elliptic three-manifolds, each space is determined up to local isometry by the $3$-multiset consisting of its Christoffel symbols (with respect to a Milnor Frame) and the order of its fundamental group.
\end{prop}

Applying Equations~\eqref{eq:VolumeAndMetricEigenvalues} and \eqref{eq:ChristoffelToMetricEigenvalues} we discover

\begin{equation}\label{eq:VolumeAndSymmetricPolynomials}
\frac{\vol(g)^2}{(16\pi^2)^2} = \eta_1^2 \eta_2^2\eta_3^2 =   \left(\frac{\vol(g) }{16\pi^2} \right)^3 \left( P_1(\mu)P_2(\mu) - P_3(\mu) \right).
\end{equation}

\noindent 
Now, let $g_\Gamma$ be a locally homogeneous metric on the elliptic three-manifold $\Gamma \backslash S^3$ and let $g$ be the left-invariant metric on $S^3$ covering $g_\Gamma$. Then, the volume of $g_\Gamma$ has the following expression in terms of the Christoffel symbols of $g_\Gamma$ (and $g$) with respect to a Milnor frame:
\begin{equation}\label{eq:A0Elliptic3ManifoldChristoffel}
a_0(\Gamma \backslash S^3, g_\Gamma) = \vol(g_\Gamma) = \frac{16 \pi^2}{|\Gamma| (P_1(\mu)P_2(\mu) - P_3(\mu))} = \frac{16 \pi^2}{|\Gamma|(\mu_1 + \mu_2) (\mu_1 + \mu_3) (\mu_2 + \mu_3)}.
\end{equation}
Therefore, the volume of $g$ can be expressed in terms of the Chirstoffel symbols 
\begin{equation}\label{eq:VolChristoffel}
\vol(g) = |\Gamma| \vol(g_\Gamma) =  \frac{16 \pi^2}{P_1(\mu)P_2(\mu) - P_3(\mu)} = \frac{16 \pi^2}{(\mu_1 + \mu_2) (\mu_1 + \mu_3) (\mu_2 + \mu_3)}
\end{equation} 
and we conclude that for any permutation $\sigma$
$$\eta_{\sigma(1)}^2 = \frac{1}{(\mu_{\sigma(1)} + \mu_{\sigma(2)}) (\mu_{\sigma(1)} + \mu_{\sigma(3)}) }.$$

%%%%%%%%%%%%%%%%%%%%%%%
\subsection{The proof of Theorems~\ref{thm:IsospectralSets} and \ref{thm:A0A1A2SufficeScalarFlat}}\label{subsec:ProofSpectralGeometrySphericalThreeManifolds}

\begin{lem}\label{lem:P1Positive}
Let $( \Gamma \backslash S^3, g)$ be a locally homogeneous elliptic three-manifold with vector of Christoffel symbols $\mu = (\mu_1(g), \mu_2(g), \mu_3(g))$. Then, 
$P_1(\mu)$ is positive.
\end{lem}

\begin{proof}
Follows directly from Equation~\eqref{eq:ChristoffelToMetricEigenvalues} and the fact that the metric eigenvalues are all positive.
\end{proof}

\begin{lem}\label{lem:SymmetricPolynomialsHeatInvariants}
Let $( \Gamma \backslash S^3, g)$ be a locally homogeneous elliptic three-manifold with first four heat invariants $a_0 \equiv a_0(\Gamma \backslash S^3, g)$, $a_1 \equiv a_1(\Gamma \backslash S^3, g) $, $a_2 \equiv a_2(\Gamma \backslash S^3, g)$ and $a_3 \equiv a_3(\Gamma \backslash S^3, g)$. Let $A= A(a_0, a_1) \equiv \frac{3a_1}{a_0}=3\Scal$, $B= B(|\Gamma|, a_0) \equiv \frac{16\pi^2}{a_0 | \Gamma |}$, $C= C(a_0, a_1, a_2) \equiv \frac{27a_1^2-30a_0 a_2}{4a_0^2}$, and $D= D(a_0, a_3) \equiv \frac{7!a_3}{a_0}$.  Then 
\begin{enumerate}
\item $P_2=A,$
\item $P_3=AP_1-B$,
\item $P_1 P_3=C$,
\item $AP_1^2-BP_1-C=0$,
\item $(192C-288A^2)P_1^2+(480AB)P_1+(184A^3-576AC-240B^2-D)=0$.
\end{enumerate}
\end{lem}

\begin{proof} 
The first equation is derived from Equation~\eqref{eq:A1Christoffel} of Theorem~\ref{thm:HeatInvariantsSymmetricPolysChristoffel}.  Combining the first equation and Equation~\eqref{eq:A0Elliptic3ManifoldChristoffel} we derive the second expression. The third equation is derived from Equation~\eqref{eq:A2Christoffel} of Theorem~\ref{thm:HeatInvariantsSymmetricPolysChristoffel}. The fourth equation is derived by solving for $P_3$ in the third equation and then substituting into the second expression. The fifth equation is derived from Equation~\eqref{eq:A3Christoffel} of Theorem~\ref{thm:HeatInvariantsSymmetricPolysChristoffel} by first substituting using the three equalities $P_2=A$, $\frac{7!a_3}{a_0}=D$, and $P_1P_3=C$, and then substituting $P_3=AP_1-B$.
\end{proof}

\begin{cor}\label{cor:DegenerateRicci}
Let $(\Gamma \backslash S^3, g)$ be a locally homogeneous elliptic three-manifold. The metric $g$ has degenerate Ricci tensor if and only if $C$ (equivalently $P_3$) is zero.
\end{cor}

\begin{proof}
The eigenvalues $\nu_1(g)$, $\nu_2(g)$ and $\nu_3(g)$ of the Ricci tensor are related to the Christofel symbols $\mu_1(g)$, $\mu_2(g)$ and $\mu_3(g)$ by $\nu_{\sigma(1)} = 2\mu_{\sigma(2)}\mu_{\sigma(3)}$, for any permutation $\sigma$. Therefore, we see that $g$ has degenerate Ricci tensor if and only if one of  its associated Christoffel symbols $\mu_1(g)$, $\mu_2(g)$ or $\mu_3(g)$ is zero, which is equivalent to $P_3(\mu)$ being zero. The result now follows from Lemmas~\ref{lem:SymmetricPolynomialsHeatInvariants}(3) and \ref{lem:P1Positive}.
\end{proof}

\begin{lem}\label{calculus}
With the notation as in Lemma~\ref{lem:SymmetricPolynomialsHeatInvariants}, if $A>0$ and $B^2>4A^3$, then $$\left( \frac{B-\sqrt{B^2-4A^3}}{2A}\right)^2<A.$$
\end{lem}

\begin{proof}
For $x>2\sqrt{A^3}$, the function $$F(x)=\frac{x-\sqrt{x^2-4A^3}}{2A}$$ is differentiable with negative derivative.  As $F(x)>0$ when $x>2\sqrt{A^3}$, $F(B)^2<F(2\sqrt{A^3})^2=A$.
\end{proof}

\begin{lem}\label{lem:P1HeatInvariants}
Let $(\Gamma \backslash S^3, g)$ be a locally homogeneous elliptic three-manifold with first four heat invariants $a_0 \equiv a_0(\Gamma \backslash S^3, g)$, $a_1 \equiv a_1(\Gamma \backslash S^3, g) $, $a_2 \equiv a_2(\Gamma \backslash S^3, g)$ and $a_3 \equiv a_3(\Gamma \backslash S^3, g)$. Then $P_1(\mu)$ is determined by $|\Gamma|$, $a_0$, $a_1$, $a_2$, and $a_3$. Concretely, let $A= A(a_0, a_1)$, $B = B=(|\Gamma|, a_0)$, $C= C(a_0,a_1,a_2)$ and $D = D(a_0, a_3)$ be as in Lemma~\ref{lem:SymmetricPolynomialsHeatInvariants}.  Then $A$, $B$, $C$, and $D$ determine $P_1(\mu)$.  Moreover,

\begin{enumerate}
\item If $A =0$ (i.e., the space is scalar flat), then $P_1(\mu) = -\frac{C}{B}$;
\item If $A< 0$ (i.e., the space has negative scalar curvature), then $C< 0$ and $P_1(\mu) = \frac{B-\sqrt{B^2+4AC}}{2A}$;
\item If $A > 0$ and $C \geq 0$, then $P_1(\mu) = \frac{B+\sqrt{B^2+4AC}}{2A}$;
\item If $C = 0$ (i.e., the metric has degenerate Ricci tensor), then $P_1 = \frac{B}{A}$ 
\end{enumerate}
In particular, $P_1(\mu)$ is determined by $|\Gamma|$, $a_0$, $a_1$, and $a_2$ in all cases except possibly when $A>0$ and $C<0$.
\end{lem}

\begin{proof} 
Throughout, it will be useful to recall that $P_1$ and $B$ are both positive.

First, we observe that in the case when $A=0$, Lemma~\ref{lem:SymmetricPolynomialsHeatInvariants}-(4) determines $P_1 = -\frac{C}{B}$, as required. 

For the case where $C=0$, Lemma~\ref{lem:SymmetricPolynomialsHeatInvariants}(3) implies $P_3$ is zero, and Lemma~\ref{lem:SymmetricPolynomialsHeatInvariants}(2) implies $P_1 = \frac{B}{A}$. This establishes the first and fourth assertions of the Lemma.

For the remainder of the argument, we assume $A \neq 0$. Then, by Lemma~\ref{lem:SymmetricPolynomialsHeatInvariants}(4), $P_1$ is a positive real root of the quadratic polynomial $$q_1(x) = Ax^2-Bx-C$$ and we conclude 
$$P_1=\frac{B+\sqrt{B^2+4AC}}{2A}\,\,\,\,\, \text{or}\,\,\,\,P_1=\frac{B-\sqrt{B^2+4AC}}{2A}.$$
Consequently, the discriminant of $q_1(x)$ satisfies $B^2 + 4AC \geq 0$. We now demonstrate the validity of assertions two and three in the Lemma.

If $A<0$, then the only possible positive root of $q_1(x)$ is $\frac{B-\sqrt{B^2+4AC}}{2A}$ with $C <0$. Therefore, when $A <0$, we conclude that $C <0$ and 
$$P_1(\mu) = \frac{B-\sqrt{B^2+4AC}}{2A}.$$
Similarly, when $A >0$ and $C \geq 0$, $q_1(x)$ has exactly one positive root and we conclude
$$P_1(\mu) = \frac{B + \sqrt{B^2+4AC}}{2A}.$$  

It remains to demonstrate that when $A >0$ and $C <0$, $P_1(\mu)$ is determined by $A$, $B$, $C$, and $D$.  If the discriminant $B^2+4AC=0$, then $P_1(\mu)=\frac{B}{2A}$. Now assume that $B^2 + 4AC >0$ and both roots of $q_1(x)$ are positive and distinct. To help determine which of these roots is equal to $P_1$, we notice that by Lemma~\ref{lem:SymmetricPolynomialsHeatInvariants}(5), $P_1$ is also a root of the polynomial $$q_2(x) = (192C-288A^2)x^2 + (480AB) x + (184A^3 -576AC -240B^2 -D).$$  If $q_2(x)$ is not a scalar multiple of $q_1(x)$, then $P_1(\mu)$ is the unique common root. Now, suppose $q_2(x)$ is a scalar multiple of $q_1(x)$. Then, comparing the coefficients of $x$, we find $q_2(x) = -480Aq_1(x)$. Next, comparing the quadratic terms of this expression, we find $C = -A^2$ and obtain
$$B^2-4A^3=B^2+4AC>0.$$
Now, since the cubic polynomial $x^3 +P_1(\mu)x^2 + P_2(\mu)x + P_3(\mu)$ has three real roots (i.e., $\mu_1$, $\mu_2$ and $\mu_3$) we see that its discriminant $\Delta$  is nonnegtive:
$$\Delta \equiv P_1^2 P_2^2 - 4 P_2^3 - 4 P_1^3 P_3 + 18 P_1 P_2 P_3 - 27 P_3^2 \geq 0.$$
Using Lemma~\ref{lem:SymmetricPolynomialsHeatInvariants}, we may substitute $P_2(\mu) = A$ and $P_1(\mu)P_2(\mu) = C = -A^2$ to obtain
$$ 0\leq\frac{\Delta \cdot  P_1^2}{A^2}=5P_1^4-22A P_1^2-27A^2=(P_1^2+A)(5P_1^2-27A),$$
which holds if and only if $P_1^2\geq \frac{27A}{5}>A$.  Therefore, by Lemma~\ref{calculus}, $$P_1=\frac{B+\sqrt{B^2+4AC}}{2A} = \frac{B+\sqrt{B^2-4A^3}}{2A}.$$ 
\end{proof}

\begin{thm}\label{thm:FundamentalGroupSameSize}
Let $(\Gamma \backslash S^3, g)$  and $(\Gamma' \backslash S^3, g')$ be locally homogeneous elliptic three-manifolds such that $|\Gamma| = |\Gamma'|$ and $a_j(\Gamma \backslash S^3, g)= a_j(\Gamma' \backslash S^3, g')$ for $j = 0,1,2,3$. Then, $(\Gamma \backslash S^3, g)$  and $(\Gamma' \backslash S^3, g')$ are locally isometric. In the event that 
$(\Gamma \backslash S^3, g)$ has $A \leq 0$ or $A > 0$ with $C\geq 0$, then the conclusion holds when $|\Gamma| = |\Gamma'|$ and $a_j(\Gamma \backslash S^3, g)= a_j(\Gamma' \backslash S^3, g')$ for $j = 0,1,2$.
\end{thm}

\begin{proof}
For any locally homogeneous elliptic three-manifold $(\Gamma  \backslash S^3, g)$ let $a_j \equiv a_j(\Gamma  \backslash S^3, g)$ be the $j$-th heat invariant and $\mu = \mu(g) = (\mu_1(g), \mu_2(g), \mu_3(g))$ be the vector of Chrtistoffel symbols with respect to some Milnor frame. Lemmas~\ref{lem:SymmetricPolynomialsHeatInvariants} and \ref{lem:P1HeatInvariants} imply that $P_1(\mu)$ and $P_2(\mu)$ are determined by $|\Gamma|$, $a_0$, $a_1$, $a_2$, and $a_3$ and that $P_2(\mu)$ is determined by $a_0$ and $a_1$. Therefore, for $(\Gamma \backslash S^3, g)$  and $(\Gamma' \backslash S^3, g')$ as in the statement of the theorem, we have $P_j(\mu) = P_j(\mu')$ for $j =1,2,3$. From Lemma~\ref{lem:SymmetricPolynomialsUsefulFact}, it follows that the $3$-multiset of Christoffel symbols is identical for these two spaces and we conclude via Proposition~\ref{prop:DeterminedByChristoffelSymbols} that these locally homogeneous elliptic three-manifolds are locally isometric. 

Under the hypotheses of the second statement, we see that $P_1(\mu)$ and, hence, $P_3(\mu)$ do not depend on $a_3$. This completes the proof.
\end{proof}

\begin{proof}[Proof of Theorem~\ref{thm:IsospectralSets}]
Follows immediately from Theorems~\ref{thm:FundamentalGroupSameSize} and Theorem~\ref{thm:LocHomogEllipticThreeManifolds}. 
\end{proof}

\begin{proof}[Proof of Theorem~\ref{thm:A0A1A2SufficeScalarFlat}]
Checking the definition of $A$ and $C$ in Lemma~\ref{lem:P1HeatInvariants}, we see the first statement of the theorem is a restatement of the last statement of Theorem~\ref{thm:FundamentalGroupSameSize}. The second statement of the theorem follows from the fact that on $S^3$ (respectively, $\SO(3)$) locally isometric left-invariant metrics are isometric.
\end{proof}

%%%%%%%%%%%%%%%%%%%%%%%
\subsection{Proof of Theorem~\ref{thm:AudibillityLocalGeometry}} 
In \cite[Thm. 1.1(2)]{LSS} we established that the non-degeneracy of the Ricci tensor is audible among elliptic three-manifolds. The following shows isospectral elliptic three-manifolds with degenerate Ricci tensor are locally isometric and have fundamental groups of the same order.  

\begin{lem}\label{lem:DegenerateRicci}
Let $(\Gamma_1 \backslash S^3, g_1)$ and $(\Gamma_2 \backslash S^3, g_2)$ be two locally homogeneous elliptic three-manifolds with $a_j(g_1) = a_j(g_2)$ for $j = 0,1,2,3$. If $g_1$ (and, hence, $g_2$) has degenerate Ricci tensor, then $|\Gamma_1| = |\Gamma_2 |$ and $g_1$ and $g_2$ are locally isometric.
\end{lem}

\begin{proof}
It will be useful to recall that for any real numbers $\alpha, \beta, \gamma$ we have 
$$(x+\alpha)(x+\beta)(x+\gamma) = x^3 + P_1(\alpha, \beta, \gamma)x^2 + P_2(\alpha, \beta, \gamma)x + 
P_3(\alpha, \beta, \gamma),$$
where $P_j(x,y,z)$ is the $j$-th elementary symmetric polynomial in three variables.

For $j=1,2$, let $\mu(g_j) = (\mu_1(g_j), \mu_2(g_j), \mu_3(g_j))$ denote the Christoffel symbols of $(\Gamma_j \backslash S^3, g_j)$ with respect to some Milnor frame. The hypotheses of the lemma imply $P_3(\mu(g_1)) = P_3(\mu(g_2)) = 0$. Without loss of generality we may assume $\mu_1(g_1)$ and  $\mu_1(g_2)$ are both zero. Then, by Theorem~\ref{thm:HeatInvariantsSymmetricPolysChristoffel} and Equation~\eqref{eq:A0Elliptic3ManifoldChristoffel}, for $j=1,2$ we have 
\begin{equation}\label{eqn:DegenerateA0}
a_0(g_j) = \frac{16 \pi^2}{|\Gamma_j| P_1(\mu(g_j)) P_2(\mu(g_j))} 
\end{equation}
\begin{equation}\label{eqn:DegenerateA1}
a_1(g_j) = \frac{16\pi^2}{3 |\Gamma_j| P_1(\mu(g_j))}
\end{equation}
\begin{equation}\label{eqn:DegenerateA2}
a_2(g_j) = \frac{a_0}{10}P_2(\mu(g_j))^2
\end{equation}

Since $a_0(g_j)$ and $P_1(\mu(g_j))$ are both positive, Equation~\eqref{eqn:DegenerateA0} implies $P_2(\mu(g_j))$ is positive. Therefore, it follows from Equation~\eqref{eqn:DegenerateA2} and the fact that $a_2(g_1) = a_2(g_2)$ that $P_2(\mu(g_1)) = P_2(\mu(g_2))$. 
Now, since $P_3(\mu(g_1)) = P_3(\mu(g_2)) = 0$, Equation~\eqref{eq:A3Christoffel} shows that for $j=1,2$
$$a_3(g_j) = \frac{a_0(g_j)}{7!}\left( 184P_2(\mu(g_j))^3 - 48P_1(\mu(g_j))^2P_2(\mu(g_j))^2\right).$$
Finally, since $P_2(\mu(g_1)) = P_2(\mu(g_2))$ and since $P_1(\mu(g_1))$ and $P_1(\mu(g_2))$  are both positive, the equality $a_3(g_1) = a_3(g_2)$ implies $P_1(\mu(g_1)) = P_1(\mu(g_2))$.
Therefore, the multisets $$[\mu_1(g_1), \mu_2(g_1), \mu_3(g_1)]\,\,\,\,\,\,\,\, \text{and}\,\,\,\,\,\,\,\, [\mu_1(g_2), \mu_2(g_1), \mu_3(g_1)]$$ are equal and we conclude $g_1$ and $g_2$ are locally isometric. The statement about the order of the fundamental groups follows from Equation~\eqref{eqn:DegenerateA1}.
\end{proof}

The proof of Theorem~\ref{thm:AudibillityLocalGeometry} is now easily obtained.

\begin{proof}[Proof of Theorem~\ref{thm:AudibillityLocalGeometry}]
Let $(M_1 = \Gamma_1 \backslash S^3, g_1)$ and $(M_2 = \Gamma_2 \backslash S^3, g_2)$ be as in the statement of the theorem and let $\Ric_1$ and $\Ric_2$ denote their respective Ricci tensors. By Theorem~\ref{thm:FundamentalGroupSameSize}, the groups $\Gamma_1$ and $\Gamma_2$ have distinct orders. Now, by \cite[Theorem 1.1]{LSS}, $\Ric_1$ and $\Ric_2$ are both degenerate or both non-degenerate. Since the manifolds are assumed to have distinct local geometry, it follows from Lemma~\ref{lem:DegenerateRicci} that $\Ric_1$ and $\Ric_2$ must be non-degenerate and it follows from \cite[Theorem 1.1]{LSS} that $g_1$ and $g_2$ must be sufficiently far from a metric of positive curvature. 
\end{proof}

%%%%%%%%%%%%%%%%%%%%%%%%
%%%%%%%%%%%%%%%%%%%%%%%%
\section{Recovering a metric from its curvature and volume}\label{sec:ExtraStuff}

As we noted in the introduction, homogeneous surfaces with identical curvature tensors---see p.~\pageref{IdenticalCurvature} for definition---must be locally isometric. However, there are continuous families of left-invariant metrics on $S^3$ (respectively $\SO(3)$ and $\SL(2, \R)$) with identical curvature tensor that are pairwise locally non-isometric \cite{Lastaria, ScWo1, ScWo2}. We conclude this article by showing that our framework allows one to observe that these ambiguities can be resolved in the presence of volume. 
 
We will say that two locally homogeneous elliptic three-manifolds $(\Gamma_1 \backslash S^3, g_1)$ and $(\Gamma_2 \backslash S^3, g_2)$ are said to be \emph{isocurved} if (up to permutation) they have identical principal curvatures. Then, the Ricci tensors of two isocurved locally homogeneous elliptic three-manifolds are either both degenerate or both non-degenerate.

\begin{thm}\label{thm:Isocurved}
Let $(\Gamma_1 \backslash S^3, g_1)$ and $(\Gamma_2 \backslash S^3, g_2)$ be isocurved locally homogeneous elliptic three-manifolds.  
\begin{enumerate}
\item If their respective Ricci tensors are non-degenerate, then $(\Gamma_1 \backslash S^3, g_1)$ and $(\Gamma_2 \backslash S^3, g_2)$ are locally isometric. In particular, in the event that $\Gamma_1 = \Gamma _2$ is not conjugate to $\Gamma_{q;1,p}$, where $q \geq 3$, we may conclude that the spaces are isometric.

\item If their respective Ricci tensors are degenerate and $|\Gamma_1| = |\Gamma_2|$, then $(\Gamma_1 \backslash S^3, g_1)$ and $(\Gamma_2 \backslash S^3, g_2)$ are locally isometric if and only if they have the same volume. In particular, in the event that $\Gamma_1 = \Gamma_2$ is not conjugate to $\Gamma_{q;1,p}$, where $q \geq 3$, we may replace ``locally isometric'' with ``isometric.''
\end{enumerate}
\end{thm}

\begin{proof}
Throughout, we let $\mu_1(g_j)$, $\mu_2(g_j)$ and $\mu_3(g_j)$ denote the Christoffel symbols of $(\Gamma_j \backslash S^3, g_j)$, for $j =1,2$.

\begin{enumerate}
\item Since $g_1$ and $g_2$ each have non-degenerate Ricci tensor, the associated Christoffel symbols are all non-zero. Also, from Equation~\eqref{eq:ChristoffelToMetricEigenvalues}, we conclude at least two of the Christoffel symbols are positive. Without loss of generality (and after possibly relabelling the symbols), we may assume $\mu_2(g_j)$ and $\mu_3(g_j)$ are positive, for $j =1,2$. Using Equation~\eqref{eqn:PrincipalCurvaturesChristoffelSymbols} we find
$$\mu_1(g_1)^2 = \frac{\mu_1(g_1)\mu_2(g_1)\mu_1(g_1)\mu_3(g_1)}{\mu_2(g_1)\mu_3(g_1)} =  \frac{\mu_1(g_2)\mu_2(g_2)\mu_1(g_2)\mu_3(g_2)}{\mu_2(g_2)\mu_3(g_2)} = \mu_1(g_2)^2.$$
Similarly, $\mu_2(g_1)^2 = \mu_2(g_2)^2$ and $\mu_3(g_1)^2 = \mu_3(g_2)^2$. Then, because we have assumed $\mu_2(g_j)$ and $\mu_3(g_j)$ are both positive, for $j =1,2$, we have   $\mu_k(g_1) = \mu_k(g_2)$ for $k = 2,3$. But, since Equation~\eqref{eqn:PrincipalCurvaturesChristoffelSymbols} implies $\mu_1(g_1)\mu_2(g_1) = \mu_1(g_2) \mu_2(g_2)$, we also have $\mu_1(g_1) = \mu_1(g_2)$. Therefore, by Proposition~\ref{prop:DeterminedByChristoffelSymbols}, $g_1$ and $g_2$ are locally isometric metrics. The second part of the statement follows from the first by applying the classification of locally homogeneous elliptic three-manifolds (Theorem~\ref{thm:LocHomogEllipticThreeManifolds}). 

\item Since both $g_1$ and $g_2$ both have degenerate Ricci tensors (equivalently, $P_3$ vanishes on $\mu(g_1) = (\mu_1(g_1), \mu_2(g_1), \mu_3(g_1))$  and $\mu(g_2) = (\mu_1(g_2), \mu_2(g_2), \mu_3(g_2))$), each metric has at least one Christoffel symbol that is zero. Without loss of generality, assume $\mu_1(g_1) = \mu_1(g_2) =0$. Since $g_1$ and $g_2$ are isocurved, Equation~\eqref{eqn:PrincipalCurvaturesChristoffelSymbols} implies  
$$P_2(\mu(g_1)) = \mu_2(g_1) \mu_3(g_1) = \mu_2(g_2) \mu_3(g_2) = P_2(\mu(g_2)).$$
Also, observe that, since $|\Gamma_1| = |\Gamma_2|$ and $\vol(g_1) = \vol(g_2)$, the volumes of the universal covering metrics $\tilde{g}_1$ and $\tilde{g}_2$ are equal. Therefore, we may use Equation~\eqref{eq:VolumeAndSymmetricPolynomials} to conclude that 
$$P_1(\mu(g_1)) = P_1(\mu(g_2)).$$
Then, by Proposition~\ref{prop:DeterminedByChristoffelSymbols}, $g_1$ and $g_2$ are locally isometric.
The second part of the statement follows from the first by applying the classification of locally homogeneous elliptic three-manifolds (Theorem~\ref{thm:LocHomogEllipticThreeManifolds}).
\end{enumerate}

\end{proof}

\begin{proof}[Proof of Theorem~\ref{thm:VolCurvature}]
Recall that two Riemannian manifolds $(M,g)$ and $(M', g')$ are said to have identical curvature tensors if for each $p \in M$ and $p' \in M'$ there is a linear isometry $F: (T_pM, g_p) \to (T_{p'}M', g'_{p'})$ such that $F^*R'_{p'} = R_p$. Now, let $M = \Gamma \backslash S^3$ be a spherical three-manifold and suppose $g$ and $g'$ are locally homogeneous metrics on $M$ with identical curvature tensor. Then,  $g$ and $g'$ are isocurved metrics and $g$ has a non-degenerate Ricci tensor if and only if $g'$ has a non-degenerate Ricci tensor. Therefore, when these metrics have the same volume, Theorem~\ref{thm:Isocurved} implies $g$ and $g'$ are locally isometric. The last statement then follows form Theorem~\ref{thm:LocHomogEllipticThreeManifolds}(3).
\end{proof}

We conclude this article by noting that among left-invariant metrics with degenerate Ricci tensor $S^3$ admits a unique (up to isometry) metric with a prescribed volume and (positive) scalar curvature. First, we provide a lemma.

\begin{lem}\label{lem:ScalVolDegenerateRicci}
If $(\Gamma \backslash S^3, g)$ is a locally homogeneous elliptic three-manifold with degenerate Ricci tensor, then
\begin{enumerate}
\item $\Scal(g) >0$, and
\item $\left(\frac{32 \pi^2}{\Scal(g) \vol(g)|\Gamma|} \right)^2 -2 \Scal(g) \geq 0.$
\end{enumerate}
\end{lem}

\begin{proof}
Throughout this argument we let $S = \Scal(g)$ and $V = \vol(g)$. The hypothesis that $g$ has degenerate Ricci tensor is equivalent to $P_3(\mu) =0$, where $\mu = (\mu_1, \mu_2, \mu_3)$ is the vector of Christoffel symbols of $g$. Therefore, at least one of the $\mu_j$'s is zero. Without loss of generality, we may assume $\mu_1 =0$. Then, by Equation~\eqref{eq:VolumeAndSymmetricPolynomials} applied to the universal covering metric $\tilde{g}$ which has $\vol(\tilde{g}) = \vol(g) |\Gamma|$,
$$\mu_2 + \mu_3 = P_1(\mu) = \frac{32 \pi^2}{SV|\Gamma|}.$$
And, since $P_1(\mu)$ is positive (see Lemma~\ref{lem:P1Positive}), we determine that $S$ is positive.

To verify the second statement, we solve the equation above for $\mu_2$  and combine with $\frac{1}{2}S =P_2(\mu) = \mu_2\mu_3$ to obtain
$$\mu_3^2 - \frac{32 \pi^2}{SV|\Gamma|} \mu_3 + \frac{1}{2}S = 0.$$
Therefore, $$\mu_3 = \frac{\frac{32\pi^2}{SV |\Gamma|} \pm \sqrt{\left(\frac{32\pi^2}{SV|\Gamma|}\right)^2 -2S}}{2}$$ and the second statement follows from the fact that $\mu_3$ is a real number.
\end{proof}

\begin{thm}
For any $S, V >0$ satisfying $\left(\frac{32 \pi^2}{SV} \right)^2 -2 S \geq 0$, up to isometry, there is a unique left-invariant metric $g$ with degenerate Ricci tensor on $S^3$ (respectively, $\SO(3)$) satisfying $S = \Scal(g)$,  $V = \vol(g)$.
\end{thm}

\begin{proof}
Fix $S, V >0$ and suppose that $g$ is a metric with degenerate Ricci tensor on $S^3$ (respectively, $\SO(3)$) such that $\Scal(g) = S$ and $\vol(g) =V$, then following the proof of the previous lemma we find the $3$-multiset of Christoffel symbols of $g$ is given by 
$$\left[ 0, \frac{\frac{32\pi^2}{SV|\Gamma|} - \sqrt{\left(\frac{32\pi^2}{SV|\Gamma|}\right)^2 -2S}}{2}, \frac{\frac{32\pi^2}{SV|\Gamma|} + \sqrt{\left(\frac{32\pi^2}{SV|\Gamma|}\right)^2 -2S}}{2}\right],$$ 
where $|\Gamma|$ is one (respectively, two).
Since this $3$-multiset consists of zero and two positive numbers (counting multiplicities), the volume assumption and Equation~\eqref{eq:ChristoffelToMetricEigenvalues} show that it determines a $3$-multiset $[x,y,z]$ of positive numbers and, hence, a unique left-invariant metric on $S^3$ (respectively, $\SO(3)$). A computation shows this metric has the desired volume and scalar curvature.
\end{proof}

 %%%%%%%%%%%%%%%%%%%%%%%%
%%%    BIBLIOGRAPHY    %%%%%%%%%%
%%%%%%%%%%%%%%%%%%%%%%%%

\end{document}